\documentclass{amsart} 

\usepackage{amsmath}
\usepackage{amsthm}
\usepackage{amsbsy}
\usepackage{amssymb}
\usepackage{mathtools}
\usepackage{calrsfs}
\usepackage{comment}
\usepackage[shortlabels]{enumitem}
\usepackage[normalem]{ulem}
\usepackage{amsrefs}
\usepackage{tikz}
\usetikzlibrary{positioning,patterns,calc}
\usepackage{mathrsfs}
\DeclareMathAlphabet{\mathpzc}{OT1}{pzc}{m}{it}
\usepackage[utf8]{inputenc}
\usepackage{hyperref}

\newcommand{\real}{\mathbb{R}} 

\def\R{\mathbb{R}}
\def\C{\mathbb{C}}

\def\N{\mathbb{N}}

\newcommand{\sffM}{II} 
\newcommand{\scalar}{R} 
\newcommand{\Index}{\textup{Index}} 
\newcommand{\harmonicvfb}{\mathcal{H}^1_T} 
\newcommand{\Half}{\mathcal{H}} 

\newtheorem{theorem}{Theorem}[section]
\newtheorem{lemma}[theorem]{Lemma}
\newtheorem{remark}[theorem]{Remark}
\newtheorem{proposition}[theorem]{Proposition}
\newtheorem{corollary}[theorem]{Corollary}

\newtheorem{definition}[theorem]{Definition}
\newtheorem*{theorem*}{Theorem}

\newtheorem{claim}{\texttt{Claim}}
\makeatletter
\@addtoreset{claim}{theorem}
\makeatother

\theoremstyle{plain}
\newtheorem{innercustomthm}{Theorem}
\newenvironment{customthm}[1]
  {\renewcommand\theinnercustomthm{#1}\innercustomthm}
  {\endinnercustomthm}
  
 \theoremstyle{plain}

\usepackage{enumitem}
\usepackage{lipsum}
\setlist[itemize]{leftmargin=*}

\def\bpf{\begin{proof}}
\def\epf{\end{proof}}
\def\be{\begin{equation}}
\def\ee{\end{equation}}
\def\bea{\begin{eqnarray}}
\def\eea{\end{eqnarray}}
\def\bt{\begin{theorem}}
\def\et{\end{theorem}}
\def\bl{\begin{lemma}}
\def\el{\end{lemma}}
\def\br{\begin{remark}}
\def\er{\end{remark}}
\def\bc{\begin{corollary}}
\def\ec{\end{corollary}}
\def\bd{\begin{definition}}
\def\ed{\end{definition}}
\def\bp{\begin{proposition}}
\def\ep{\end{proposition}}

\def\ll{\langle}
\def\rl{\rangle}

\usepackage{etoolbox}
\makeatletter
\patchcmd{\@maketitle}
  {\ifx\@empty\@dedicatory}
  {\ifx\@empty\@date \else {\vskip3ex \centering\footnotesize\@date\par\vskip1ex}\fi
   \ifx\@empty\@dedicatory}
  {}{}
\patchcmd{\@adminfootnotes}
  {\ifx\@empty\@date\else \@footnotetext{\@setdate}\fi}
  {}{}{}
\makeatother

\title[Capillary surfaces: stability, index and curvature estimates]{Capillary surfaces: stability, index and curvature estimates}
\author{Han Hong}
\address{Yau Mathematics Science Centre, Tsinghua University, Beijing, China, 100084}
\email{hh0927@mail.tsinghua.edu.cn}

\author{Artur B. Saturnino}
\address{Department of Mathematics, University of Pennsylvania, Philadelphia, PA 19104,USA}
\email{bsatur@sas.upenn.edu}

\begin{document}
\maketitle

\begin{abstract}
    In this paper we investigate the connection between the index and the geometry and topology of capillary surfaces. We prove an index estimate for  compact capillary surfaces immersed in general 3-manifolds with boundary. We also study noncompact capillary surfaces with finite index and show that, under suitable curvature assumptions, such surface is conformally equivalent to a compact Riemann surface with boundary, punctured at finitely many points. We then prove that a weakly stable capillary surface immersed in a half-space of $\mathbb{R}^3$ which is minimal or has a contact angle less than or equal to $\pi/2$ must be a half-plane. Using this uniqueness result we obtain curvature estimates for strongly stable capillary surfaces immersed in a 3-manifold with bounded geometry.
\end{abstract}

\section{Introduction}

Let $M$ be a smooth three-dimensional Riemannian manifold with smooth boundary. We are interested in capillary surfaces, namely, two-sided minimal or constant mean curvature (CMC) surfaces in $M$ with boundary  in $\partial M$ that intersect $\partial M$ at a constant angle $\theta\in (0,\pi)$. When $\theta=\pi/2$, we usually call such surfaces free boundary surfaces. From the variational point of view, capillary surfaces are critical points of a certain energy functional with respect to compactly supported volume-preserving variations. This energy functional involves the area of the surface, the wet area in the boundary of the ambient space and the angle of contact between the surface and the boundary of the ambient space (see definition in Section \ref{sec: prelim}). In fact, capillary surfaces in subdomains of $\mathbb{R}^3$ model incompressible liquids inside containers in the absence of gravity. In this context the contact angle only depends on the strength of the interactions between the liquid, the container and the ambient gas. For interested readers, Finn's book \cite{Finn} is recommended as a good survey about capillary surfaces.

The weak index and the strong index of a capillary surface are natural quantities associated to the second variation of the energy functional. These quantities measure the dimension of the space of deformations of the surface that decrease energy to second order. The difference between these indices involves the deformations that are considered: In computing the weak index one only considers volume-preserving deformations while in computing the strong index one considers all deformations. For this reason the weak index is more natural when studying CMC surfaces while the strong index is more natural in the context of minimal surfaces. It follows from the min-max principle that the weak index is not greater than the strong index, however they can differ by at most one, and
for some complete noncompact CMC surfaces they coincide \cite{Barbosa-Berard-twisted-eigenvalue}.

As we will illustrate below, it is well-known that, in the complete and in the free boundary cases, the weak index and the strong index of a CMC or minimal surface are closely linked to the geometry and the topology of the surface. Our objective in this paper is to explore this link in the case of capillary surfaces.

\subsection{Index estimates for compact capillary surfaces} We say that a capillary surface is weakly (resp. strongly) stable if its weak (resp. strong) index is zero. It is natural to try to classify all weakly or strongly stable capillary surfaces in a manifold $M$. We first only consider compact capillary surfaces.

Two interesting examples are the cases when $M$ is an unit ball $\mathbb{B}$ or a (closed) half-space of $\mathbb{R}^3$. If $M$ is the unit ball, this classification question has been completely solved (see for example \cite{Wang-Xia-Uniqueness-of-stable}), any weakly stable capillary surface in the unit ball must be totally umbilical, i.e., they are spherical caps or equatorial disks.  If $M$ is a half-space of $\mathbb{R}^3$, any  embedded compact capillary surfaces in $M$ must be rotational symmetric along some axis that is perpendicular to the boundary of the half space due to Alexandrov's reflection method, so embedded capillary surfaces are spherical caps \cite{Wente-Alexandrov-reflection}. Furthermore, Marinov \cite{Marinov-stability-of-capillary} showed that the only weakly stable compact capillary surfaces immersed in $M$ with embedded boundary are spherical caps. Ainouz and Souam \cite{Ainouz-Souam-capillarysurfaceinslab} improved on this result by extending it to higher dimensions and showing that one only need to assume that  $\theta = \pi/2$ or that the boundary components of the surface are embedded. It is an interesting question whether or not this extra assumption can be dropped.

Index estimates for compact capillary surfaces in general 3-manifolds are also interesting to investigate. For free boundary minimal surfaces, some results about bounding index from below in terms of the topology of the surface have been proved by Sargent \cite{Pam-freeboundary-inball} and Ambrozio-Carlotto-Sharp \cite{Ambrozio-Carlotto-Sharp-freeboundaryminimalsurface}. In these papers they used harmonic one forms to construct test functions for the second variation formula of area functional. This idea was first discovered by Ros \cite{Ros-onesided-minimalsurface}, and further developed by Savo \cite{Savo-indexbounds-onsphere} and Ambrozio-Carlotto-Sharp \cite{Ambrozio-Carlotto-Sharp-closedminimalsurfacesinmanifolds} in the case when the surfaces in question are closed minimal surfaces. Aiex and the first named author applied a similar idea to the settings where the surfaces are closed CMC surfaces and free boundary CMC surfaces in a general three-dimensional Riemannian manifold \cite{AH2021} and obtained lower index bounds in terms of the topology. In particular, as a byproduct, they showed that the index of free boundary CMC surfaces in a mean convex domain of $\mathbb{R}^3$ is bounded below by $(2g-r-4)/6$ where $g$ is the genus of the surface and $r$ is the number of boundary components of the surface. This particular result was also obtained by Cavalcante-de Oliveira in \cite{Cavalcante-Oliveira-freeboundary}. Here we generalize these results to compact capillary surfaces. To be precise, we prove the following:

\begin{theorem}\label{index estimate free boundary-introduction}
Let $M$ be a $3$-dimensional oriented Riemannian manifold with boundary isometrically embedded in $\real^d$ and let $\Sigma$ be a compact capillary surface immersed in $M$ at a constant angle $\theta$ with genus $g$ and $r$ boundary components.
Suppose that every non-zero $\xi\in \harmonicvfb(\Sigma,\partial \Sigma)$ satisfies
\begin{equation*}
\begin{aligned}
\int_{\Sigma} & \sum_{i=1}^2|\sffM_M(e_i,\xi)|^2+|\sffM_M(e_i,\star\xi)|^2\ dA
-\int_\Sigma R_{M}|\xi|^2\ dA\\ &-\int_{\partial\Sigma} \frac{2}{\sin\theta}H_{\partial M}|\xi|^2\ d\ell <\int_\Sigma H_\Sigma^2|\xi|^2\ dA+\int_{\partial\Sigma}2\cot\theta H_\Sigma|\xi|^2\ d\ell.
\end{aligned}
\end{equation*}
Then
\[\Index_w(\Sigma)\geq\frac{2g+r-1-d}{2d}.\]
\end{theorem}

In particular, in $\mathbb{R}^3$ we have
\begin{corollary}\label{particular case in R3}
Let $M$ be a domain in $\mathbb{R}^3$ with smooth boundary and let $\Sigma$ be a compact capillary surface immersed in $M$ at a constant angle $\theta$ with genus $g$ and $r$ boundary components.
Suppose that $H_{\partial M} +H_\Sigma\cos\theta \ge 0$ along $\partial \Sigma$ and that one of the
following holds:
\[H_\Sigma > 0,\] 
or 
\[ H_{\partial M} > 0\]
at some point in $\partial \Sigma$.
Then
\[\Index_w(\Sigma)\geq\frac{2g+r-4}{6}.\]
\end{corollary}

The ideas of the proof of Theorem \ref{index estimate free boundary-introduction} are essentially the same as those in \cite{AH2021}, though the calculations are more involved because of the angle $\theta$ is not necessarily $\pi/2$.

\begin{remark}\label{rmk: compact surface in half-space}
When $M$ is a closed half-space in $\mathbb{R}^3$, $H_{\partial M}+H_\Sigma\cos\theta$ reduces to $H_\Sigma\cos\theta.$ As compact capillary surfaces in a closed half-space cannot be minimal due to maximum principle, the assumption of Corollary \ref{particular case in R3} simply becomes $\theta\in(0,\pi/2].$
\end{remark}

\subsection{Noncompact capillary surfaces} 
We characterize the strong index of noncompact capillary  surfaces. More precisely, we show that strong stability is characterized by the existence of a positive Jacobi function (Proposition \ref{prop: positive jacobi field}) and that the strong index of a noncompact capillary surface is realized as the dimension of a space generated by $L^2$ eigenfunctions (Proposition \ref{prop: L2 characterization}).

Fisher-Colbrie \cite{Fischer-Colbrie-On-complete-minimal}*{Theorem 1} has shown that a two-sided complete
noncompact minimal surface in an oriented ambient space of nonnegative scalar curvature is conformally equivalent to a
compact Riemann surface with finitely many points removed. We will show that an analogous statement holds for 
noncompact capillary surfaces:
\begin{theorem}
\label{thm: Finite topology intro}
Let $M$ be an oriented Riemannian 3-manifold with smooth boundary and let $\Sigma$ be a 
noncompact capillary surface with finite index immersed in
$M$ at a constant angle $\theta$. Assume that 
$R_M + H^2_{\Sigma} \ge 0$ and that one of the following holds:
\[\partial \Sigma \text{ is compact,}\]
or 
\[H_{\partial M} + H_\Sigma\cos \theta  \ge 0 \ \text{along $\partial \Sigma$}.\]
Then $\Sigma$ is conformally equivalent to a compact
Riemann surface with boundary and finitely many points removed, each associated to an end of the surface. Moreover, 
\[\int_\Sigma R_M + H_\Sigma^2 + |A_\Sigma|^2 + \int_{\partial \Sigma} H_{\partial M} + H_\Sigma\cos \theta < \infty.\]
\end{theorem}
Let $\overline{\Sigma}$ be the compact Riemann surface that describes the topology of the noncompact capillary surface 
$\Sigma$ as in Theorem \ref{thm: Finite topology intro}. Points removed from the interior of 
$\overline{\Sigma}$ are associated to \textit{interior ends} of $\Sigma$, and points removed from the boundary of
$\overline{\Sigma}$ are associated to \textit{boundary ends} of $\Sigma$. These are described in more details in Remark \ref{rmk: description of conformal structure} and Remark \ref{rmk: description of ends}. Two corollaries of Theorem \ref{thm: Finite topology intro} are proven. Namely, Corollary \ref{corollary: capillary must be minimal} gives conditions under which noncompact capillary surfaces in domains of $\R^3$ are minimal, and Corollary \ref{cor: discription strongly stable} describes the conformal type of strongly stable capillary surfaces under suitable curvature assumptions.

Capillary surfaces in a half-space of $\R^3$ are particularly important because they serve as local models for  capillary surfaces near the boundary. This fact will be essential to the proof of our curvature bounds (Theorem \ref{curvature bound-introduction}). In particular, we will need to understand the strongly stable capillary surfaces in a half-space of $\R^3$. As observed in Remark \ref{rmk: compact surface in half-space}, there are no compact minimal capillary surfaces in a half-space of $\R^3$, so in this context it is essential to consider noncompact capillary surfaces in it.

The only strongly stable complete CMC surfaces in $\R^3$ are planes, this was first proven for two-sided minimal surfaces independently by do Carmo-Peng \cite{doCarmo-Peng}, Fischer-Colbrie-Schoen \cite{Fischer-Colbrie-Schoen-The-structure-of-complete-stable} and Pogorelov \cite{Pogorelov-stable}, later for CMC surfaces by da Silveira \cite{daSilveira-Stability-of-complete} and for one-sided minimal surfaces by Ros \cite{Ros-onesided-minimalsurface}. It is then natural to ask whether half-planes intersecting the boundary of a half-space at a constant angle $\theta$ are the only strongly stable capillary surfaces immersed in a half-space. 
When the angle is $\pi/2$, namely, the surface has free boundary, we are able to reflect the surface to get a smooth strongly stable complete noncompact CMC surface, thus the result in the case when the surface has no boundary is helpful to study the free boundary case. In particular, it has been proven by Ambrozio-Buzano-Carlotto-Sharp \cite{Ambrozio-Buzano-Carlotto-Sharp}*{Corollary 2.2} that a stable free boundary minimal surface in a half-space  must be a half-plane. When the angle is not $\pi/2$, the reflection analysis apparently fails, moreover, the boundary term in the second variation formula does not vanish. Nevertheless, we are able to show:

\begin{theorem}\label{maintheorem1}
Let $\Sigma$ be a noncompact capillary surface immersed in a half-space of $\R^3$ at constant angle $\theta$.
Assume that $H_\Sigma\cos \theta \ge 0$. Then $\Sigma$ is weakly stable if and only if it is a half-plane.
\end{theorem}

It is easy to see that the assumption $H_\Sigma \cos \theta \ge 0$ translates to $\Sigma$ has nonzero mean curvature and $\theta \in (0, \pi/ 2]$ or $\Sigma$ is minimal. We will also observe in Remark \ref{remark: regidity compact case} that there is no weakly stable noncompact 
capillary surface with compact boundary immersed in a half-space of $\R^3$, even with $H_\Sigma \cos \theta  < 0$.

In order to prove Theorem \ref{maintheorem1}, we first construct a test function involving angle $\theta$. Thanks to Theorem \ref{thm: Finite topology intro}, we can construct a cutoff function such that its multiplication with $u=\frac{1}{\sin\theta}+\cot\theta\ll \nu,-E_3\rl$ becomes a admissible test function. This function $u$ is designed to make sure that the boundary integral in the quadratic form vanishes since we have no control on the sign or size of the boundary integral when the boundary is noncompact. If the surface is not totally geodesic, we can show that this test function leads to negative second variation, thus contradicting weak stability. When $\theta=\pi/2$, the function $u$ is constant. In particular, in the case where the surface is assumed to have free boundary, these arguments give a new and direct proof to some results in \cite{Ambrozio-Buzano-Carlotto-Sharp}.

Notice that there are noncompact capillary surfaces in a half-space besides half-planes. For any angle $\theta\in (0,\pi)$ and any constant $H>0$, one can construct a noncompact capillary surface with mean curvature $H$ and contact angle $\theta$ by cutting an unduloid horizontally (by a plane orthogonal to its axis). Noncompact free boundary surfaces with any constant mean curvature can also be obtained by cutting an unduloid vertically (by a plane that contains its axis). These examples however have infinite index. Examples of minimal capillary surfaces with finite index and any contact angle $\theta \in (0, \pi)$ can be obtained by cutting a catenoid horizontally. Cutting a catenoid vertically gives a noncompact free boundary minimal surface with index one. Observe that another significant difference between cutting a catenoid horizontally or vertically is that the former has compact boundary and one interior ends while the latter has no interior ends and two boundary ends.

\subsection{Curvature estimates for strongly stable capillary surfaces} Curvature estimates for strongly stable CMC surface play an important role in the study of these surfaces. In 1983, Schoen \cite{Schoen-Estimates-for-stable-min-surf} proved that the second fundamental form of a strongly stable two-sided minimal surface immersed in a 3-manifold is bounded above by an universal constant multiplying the reciprocal of the distance to the boundary, and this constant only depends on the curvature of the ambient manifold and its covariant derivative. In particular, he proved that 
\begin{equation}\label{Schoen-curvature-estimate-in-R3}
|A_\Sigma(p)|\leq \frac{C}{d_{\Sigma}(p,\partial\Sigma)}\qquad \forall\  p \in \Sigma,
\end{equation}
where $\Sigma$ is a strongly stable two-sided minimal surface immersed in $\mathbb{R}^3$, $C>0$ is an universal constant, $A_\Sigma$ is the second fundamental form of $\Sigma$ and $d_\Sigma(p, \partial \Sigma)$ is the intrinsic distance from $p \in \Sigma$ to the boundary of $\Sigma$. Note that this curvature estimate gives a different proof to the fact that a two-sided stable complete minimal surface in $\mathbb{R}^3$ must be a plane.

Schoen's curvature estimates have many generalizations. Those more relevant to our work are the generalization to strongly stable free boundary minimal surfaces by Guang-Li-Zhou \cite{Zhou-Curvature-estimates-for-stable-fbms} and the generalization to strongly stable immersed CMC surfaces by Rosenberg-Souam-Toubiana \cite{Rosenberg-General-Curvature-Estimates}.

An estimate like \eqref{Schoen-curvature-estimate-in-R3} cannot hold for strongly stable CMC surfaces in general 3-manifolds. For instance, horospheres are strongly stable complete CMC surfaces in hyperbolic 3-space but are not totally geodesic. However, Rosenberg-Souam-Toubiana \cite{Rosenberg-General-Curvature-Estimates} showed that far away from its boundary, the second fundamental form of a strongly stable \textit{immersed} CMC surface is bounded above by an universal constant $C$ that does not depend on the ambient space or the surface divided by the square root of an absolute sectional curvature bound of the ambient space. More precisely, they obtained the following curvature estimate: There is a constant $C$ such that for any complete 3-manifold $M$ with bounded sectional curvature, i.e., $|K_M|<\Lambda$, and for any strongly stable two-sided CMC surface $\Sigma$ immersed in $M$
\begin{equation}\label{Rosenberg-curvature-estimate}
|A_\Sigma(p)|\leq \frac{C}{\min\{d_\Sigma(p,\partial\Sigma),(\sqrt \Lambda)^{-1}\}} \qquad \forall\ p \in \Sigma.
\end{equation}
In order to show this inequality Rosenberg, Souam and Toubiana use a blow-up procedure with harmonic coordinates. Using harmonic coordinates usually requires a lower injectivity radius bound, in order to avoid any assumptions on injectivity radius  they first pull a neighborhood of the ambient manifold back to its tangent space through the exponential map, where there are lower injectivity radius bounds that depend only on $\Lambda$. In the cases we are interested in, the ambient space has a boundary. Therefore some dependency of the curvature bound on the injectivity radius of the ambient manifold seems to be inevitable, otherwise the ambient manifold might ``disappear'' in the blow-up process.

When $\Sigma$ is a strongly stable \textit{edged} capillary surface (that is, a capillary surface that might have a fixed boundary in the interior of $M$, see the definition in Section \ref{subsection-edged-surfaces}) \textit{immersed} in a 3-manifold $M$ with \textit{bounded geometry} (that is, there are positive $\Lambda$ and $\iota$ such that $|K_M|\leq \Lambda, |h_{\partial M}| \le \Lambda$ and $M$ has injectivity radius at least $\iota$ in a sense defined in Section \ref{sec: curv bounds}), we show:
\begin{theorem}\label{curvature bound-introduction}
Let $\theta \in (0, \pi)$. Then there is a constant $C = C(\theta)$ such that the following holds:
Let $M$ be a 3-manifold with smooth boundary. Assume that $M$ has
curvature bounded above by $\Lambda$ and injectivity radius bounded bellow by $\iota$.
Let $\Sigma$ be a strongly stable edged capillary surface immersed in $M$ at a constant angle $\theta$. Then, 
\begin{itemize}[leftmargin=*]
    \item if
$H_\Sigma \cos \theta \ge 0$, we have
\[|A_{\Sigma}(p)|\min\{d_{\Sigma}(p, \partial\Sigma \setminus 
\partial M), \iota, (\sqrt{\Lambda})^{-1}\}\le C\]
for all $p \in \Sigma$; and
\item if 
$H_\Sigma \cos \theta < 0$, we have
\[|A_{\Sigma}(p)|\min\{d_{\Sigma}(p, \partial\Sigma \setminus 
\partial M), \iota, (\sqrt{\Lambda})^{-1}, -(H_\Sigma\cos \theta)^{-1}\}\le C\]
for all $p \in \Sigma$.
\end{itemize}
\end{theorem}
Notice that the constant $C$ only depends on $\theta$ but is independent of $\Sigma$ and $M$. Since these bounds do not depend of the area of the surface, to the best of the authors' knowledge they are also new for the special case where $\Sigma$ is an immersed strongly stable free boundary minimal surface, and hence it gives a proof to a conjecture of Guang-Li-Zhou \cite{Zhou-Curvature-estimates-for-stable-fbms}*{Conjecture 1.4}.

When $\theta \in (0, \pi/2]$ and the mean curvature $H_\Sigma$ of $\Sigma$ is large in comparisons with $\iota^{-1}$ and $\Lambda$,
it is possible to show that there is a bound on the diameter of $\Sigma$ that depends linearly on its strong index, in fact this is the content of Corollary \ref{cor:diameter bounds}.
In the special case where the ambient manifold $M$ is a half-space of $\R^3$, $\iota^{-1}$ and $\Lambda$ can be taken to be arbitrarily small, so the bounds in Corollary \ref{cor:diameter bounds} are valid for all $H_\Sigma >0$ as long as $\theta \in (0, \pi/2]$.

\subsection{Organization of the paper}
In section \ref{sec: prelim}, we give the basic definitions for compact and noncompact capillary surfaces in 3-manifolds with boundary, introduce two different notions of index, and recall some backgrounds on harmonic vector fields. In section \ref{sec: index estimates for compact capillary surfaces}, we prove index estimates (Theorem \ref{index estimate free boundary-introduction}) for compact capillary surfaces in general 3-manifolds and also prove Corollary \ref{particular case in R3}. In section \ref{sec: structure of noncompact capillary surfaces with finite index}, we first study the structure of noncompact capillary surfaces with finite index (Theorem \ref{thm: Finite topology intro}). We then prove two interesting corollaries (Corollary \ref{corollary: capillary must be minimal} and Corollary \ref{cor: discription strongly stable}) and the uniqueness result (Theorem \ref{maintheorem1}). We also prove Proposition \ref{prop: L2 characterization} which is of independent interest. Finally, in section \ref{sec: curv bounds}, we prove the curvature estimates for strongly stable capillary surfaces (Theorem \ref{curvature bound-introduction}). 

\subsection{Acknowledgements}
The authors would like to thank Prof. Ailana Fraser and Prof. Davi Maximo for their valuable comments and suggestions. The second named author would also like to thank Davi Maximo for enlightening conversations about \cite{Fischer-Colbrie-On-complete-minimal}.

\section{preliminary}\label{sec: prelim}
Let $M$ be a three-dimensional Riemannian manifold with smooth boundary and let $\Sigma$ be a surface with smooth boundary. We assume for now that $\Sigma$ is complete in the metric sense, meaning that it contains its boundary. Let $X:\Sigma\looparrowright M$ be a two-sided isometric immersion such that $X(\partial\Sigma) = \partial M \cap X(\Sigma)$. Note that this implies $\partial M \cap X(\mathring \Sigma) = \varnothing$ where $\mathring \Sigma$ is the interior of $\Sigma$. Throughout the paper, we will
no longer write the immersion $X$ if it can be suppressed without ambiguity.

Let  $X(\cdot,t):\Sigma\times (-\epsilon,\epsilon)\rightarrow M$ be a compactly supported variation such that $X(\partial \Sigma,t)\subset \partial M$ for all $t \in (-\epsilon, \epsilon)$. We denote by $\nu$ the unit normal vector field of $\Sigma$ in $M$ and denote by $N$ the outer unit normal vector field of $\partial M$. We denote the unit normal vector field of $\partial\Sigma$ in $\partial M$ by $T$ and denote the conormal of $\Sigma$ by $\eta$. When $\Sigma$ is a surface with constant mean curvature, $\nu$ can be chosen such that the mean curvature as defined later is nonnegative. At a boundary point $p\in\partial\Sigma$, $T$ can be chosen such that $\{N,T\}$ has the same orientation as $\{\eta,\nu\}$ in $T(\partial\Sigma)^\perp$ (see Figure \ref{fig:angles}).

Let $\Omega \subset \Sigma$ be the support of the variation $X(\cdot, t)$, we consider following two functionals:
\[A(t)=\int_\Omega\ dA_t,\]
\[W(t)=\int_{[0,t]\times(\partial\Sigma \cap \Omega)}X^{*}\ dA_{\partial M},\]
where $dA_t$ is the area element with respect to the metric on $X(\Sigma,t)$, and $d A_{\partial M}$ is the area element of $\partial M$. $W(t)$ is usually called wetting area functional. 

Besides, we consider a (signed) volume functional defined by
\[V(t)=\int_{[0,t]\times \Omega}X^{*}\ dV_M,\]
where $dV_M$ is the volume element of $M$. When $V(t)=V(0)$ for all $t\in (-\epsilon,\epsilon),$ we say that the variation $X(\cdot,t)$ is volume-preserving.

For a real number $\theta\in(0,\pi)$, the energy functional $E(t):(-\epsilon,\epsilon)\rightarrow \mathbb{R}$ is defined by
\[E(t)=A(t)-\cos\theta \cdot W(t).\]
It is well known (see e.g. \cites{Ros-Souam-On-stability, Wang-Xia-Uniqueness-of-stable}) that
\begin{equation}\label{first variation of volume}
    V'(0)=\int_\Sigma \ll Y,\nu\rl,
\end{equation}
\begin{equation}\label{first variation of energy}
    E'(0)= - \int_\Sigma  \ll \vec{H}_\Sigma,Y\rl +\int_{\partial\Sigma} \ll Y,\eta-\cos\theta T\rl. 
\end{equation}
Here $\vec{H}_\Sigma= - H_\Sigma \nu$ where $H_\Sigma = \operatorname{tr} A_\Sigma$ is the trace of the second fundamental form $A_\Sigma(\cdot, \cdot) = \ll D_\cdot \nu, \cdot \rl$, here $\ll \cdot,\cdot\rl$ and $D$ are the Riemmanian metric and connection in $M$, respectively, and $Y$ is the variational vector field generated by $X(\cdot,t)$, i.e., $Y=\partial_tX|_{t=0}$. The integrals are taken with respect to the area element of $\Sigma$ and the length element of $\partial \Sigma$ respectively. When there is no ambiguity about which area or length element is used we will suppress it from the the integral.

\begin{figure}
    \centering
  \begin{tikzpicture}
  
\coordinate (O) at (0,0);
\coordinate (N) at (0, -2);
\coordinate (T) at (2, 0);

\draw[->](O)--(N) node[anchor = west]  {$N$};
\draw[->](O)--(T) node[anchor = south] {$T$};
\draw[->](O)--(-30:2) coordinate (eta) node[anchor = west]  {$\eta$};
\draw[->](O)--(60:2) node[anchor = west]  {$\nu$};
\draw[->](O)--(240:2) node[anchor = east]  {$-\nu$};

\draw (1,0) arc (0:-30:1);
\draw (0,-1) arc (270:240:1);
\node at (-15:1.2){$\theta$};
\node at (250:1.2){$\theta$};

\draw[thick, dashed] (-3, 0) -- (3, 0);
\node at (-1.5, -0.2) {$\partial M$};

\clip (-3, 0) rectangle (0, 3);
\draw[thick] (-3, -5.2) circle (6);
\node at (-1.5, 0.8) {$\Sigma$};

\end{tikzpicture}
    \caption{Vectors and angles for a spherical cap in a half-space of $\R^3$.}
    \label{fig:angles}
\end{figure}
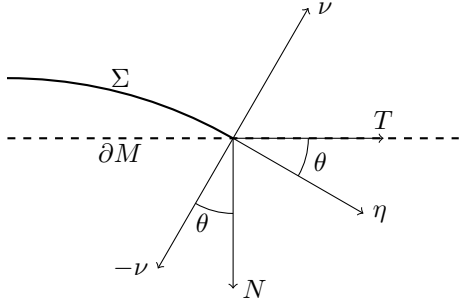

We say $\Sigma$ is capillary if it is a critical point of the energy functional $E(t)$ for any volume-preserving variation. It follows from equation $(\ref{first variation of energy})$ that $\Sigma$ is capillary if $\Sigma$ is a constant mean curvature (CMC) surface and $\Sigma$ intersects $\partial M$ at a constant angle $\theta.$ Here $\theta$ represents the angle between $N$ and $-\nu$ (same as the angle between $\eta$ and $T$). When $\theta=\pi/2$, $\Sigma$ is said to have free boundary. Note that by our definition a capillary surface is always \textit{two-sided},  complete in the sense of metric spaces, has smooth boundary that is entirely contained in $\partial M$ and can only touch $\partial M$ at its boundary.

\subsection{Second variation of compact capillary surfaces}\label{subsec: prelim compact}
Now assume that $\Sigma$ is compact. Let $X(\cdot, t)$ be a volume-preserving variation, that is, $\int_{\Sigma}u=0$ where $u= \ll Y, \nu \rl.$ Functions with zero integral are called \textit{admissible}, it is interesting to know that any admissible function is associated to a volume-preserving variation that maps $\partial \Sigma$ into $\partial M$ \cite{Ainouz-Souam-capillarysurfaceinslab}*{Proposition 2.1}. For a capillary surface $\Sigma$, the second derivative of the energy functional induces a quadratic form, called the \textit{stability operator}:
\begin{align}\label{second variation of energy}
    Q(u,u)
    &=\int_\Sigma |\nabla u|^2-(\operatorname{Ric}_M(\nu,\nu)+|A_\Sigma|^2)u^2 -\int_{\partial \Sigma}qu^2 \nonumber\\
    &=-\int_\Sigma u Ju+\int_{\partial\Sigma} u\left(\frac{\partial u}{\partial\eta}-qu\right),
\end{align}
where $J$ is the standard Jacobi operator, that is, 
\[J=\Delta+|A_\Sigma|^2+\operatorname{Ric}_M(\nu,\nu),\]
where $\operatorname{Ric}_M(\nu,\nu)$ is the Ricci curvature of $M$ along $\nu$ on $\Sigma$ and $q$ is defined by
\begin{equation}\label{eq: q std}
    q=\frac{1}{\sin\theta}h_{\partial M}(T,T)+\cot\theta A_\Sigma(\eta,\eta),
\end{equation}
where $h_{\partial M}(T,T)=\ll D_T N,T\rl$. Note that
\begin{equation}\label{eq: eta}
    \eta = \frac{1}{\sin \theta} N + \cot \theta \nu
\end{equation}
and let $H_{\partial M} = \textup{tr} h_{\partial M}$, then it follows from \eqref{eq: eta} that $q$ can be re-written as
\begin{equation}\label{eq: q geodesic curvature}
    q = \frac{1}{\sin\theta}H_{\partial M}+\cot\theta H_\Sigma - \kappa_{\partial\Sigma},
\end{equation}
where $\kappa_{\partial\Sigma}$ is the geodesic curvature of $\partial \Sigma$ in $\Sigma$, that is, $\kappa_{\partial\Sigma} = \ll D_\tau \eta, \tau \rl$ with $\tau$ unit tangent vector field in $\partial \Sigma$.

The weak (Morse) index of a compact capillary surface $\Sigma$, denoted by $\Index_w(\Sigma)$, is defined to be the maximal dimension of a subspace of $V=\{u\in C^\infty(\Sigma):\int_\Sigma u =0\}$ in which the quadratic form $Q$ is negative-definite. The weak index equals the number of negative eigenvalues of the following eigenvalue problem:
\begin{equation}\label{equations}
\begin{cases}
\tilde{J} u + \tilde \lambda u = 0 & \text{in}\ \Sigma\\
\frac{\partial u}{\partial \eta}=qu & \text{on}\ \partial \Sigma,
\end{cases}
\end{equation}
where $\tilde{J} u = J u -\frac{1}{|\Sigma|}\int_{\Sigma}J u$. The spectrum of $\tilde{J}$ consists of eigenvalues $\tilde\lambda_1\leq\tilde\lambda_2\leq\ldots$ corresponding to admissible eigenfunctions $\tilde\phi_1,\tilde\phi_2,\ldots$ in $C^\infty(\Sigma)$ (see e.g.\cite{Barbosa-Berard-twisted-eigenvalue}). Furthermore, the eigenvalues $\tilde\lambda_i$ satisfy the min-max characterization with respect to the quadratic form $Q$ when restricted to admissible functions.
That is, $\tilde \lambda_k=\inf \frac{Q(u,u)}{\int_\Sigma u^2}$, where the infimum is taken over all smooth admissible functions that are orthogonal to the first $k-1$ eigenfunctions $\tilde\phi_1,\ldots,\tilde\phi_{k-1}$.

The strong (Morse) index of $\Sigma$, denoted by $\Index_s(\Sigma)$, is the maximal
dimension of a subspace of $C^\infty(\Sigma)$ where $Q$ is negative-definite. It can be
characterized using eigenvalues in an analogous manner to the weak index by substituting
$\tilde J$ by $J$ and considering the set of all smooth functions instead of only admissible smooth functions.
It follows from the min-max principle that
\[\operatorname{Index}_w(\Sigma) \le \operatorname{Index}_s(\Sigma) \le
\operatorname{Index}_w(\Sigma) + 1.\]
 We say that $\Sigma$ is weakly (resp. strongly) stable if its weak (resp. strong) index is zero.

\subsection{Edged and noncompact capillary surfaces}\label{subsection-edged-surfaces}
An \textit{edged} surface $\Sigma$ immersed in a 3-manifold $M$ with smooth boundary is a surface whose boundary is not necessarily contained in $\partial M$, that is, $\partial\Sigma\setminus \partial M\neq \emptyset$ is possible. To be more precise, we say that a smooth surface $\Sigma$ immersed in $M$ is edged if:
\begin{itemize}
    \item $\Sigma$ does not touch $\partial M$ in its interior;
    \item $\Sigma$ contains the interior of $\Gamma = \partial M \cap \partial \Sigma$ but does not contain its \textit{edge} $\partial \Sigma \setminus \Gamma$;
    \item $\Gamma$ is smooth.
\end{itemize}

Assume $\Sigma$ is a two-sided edged surface and let $X(\cdot, t)$ be a variation supported in a compact subset of $\Sigma$, note that by our definition
this variation does not move a neighborhood of the edge. Define $W(t)$ and
$E(t)$ exactly as before, but changing $\partial \Sigma$ to $\Gamma$ everywhere. We say that $\Sigma$
 is an \textit{edged capillary surface} if it is a critical point of $E$ for any volume-preserving variation as
described above. This is equivalent to saying that $\Sigma$ has constant
mean curvature and intersects $\partial M$ at a constant angle $\theta$ along $\Gamma$.
Note that a capillary surface is an edged capillary surface with empty edge, and if $\Sigma$ is a capillary 
surface and $\Omega \subset \Sigma$ is an open set, then $\Omega$ is an edged capillary surface.

If $\Sigma$ is an edged capillary surface and $X(\cdot, t)$ is admissible, the second variation of
energy is given by a formula analogous to \eqref{second variation of energy}, the only distinction
being that $\partial \Sigma$ must be changed to $\Gamma $. When $\Sigma$ is bounded we can define the weak and the strong indices of $\Sigma$ in an analogous way as
we did for compact capillary surfaces, but now we only consider functions in $C^\infty_0(\Sigma)$, that is, smooth functions that vanish identically near the edge. 
The weak and strong indices of $\Sigma$ are also characterized by eigenfunctions, the only difference is
that the boundary conditions in \eqref{equations} change from pure Robin boundary conditions to mixed Dirichlet and Robin boundary conditions:
\[
\begin{cases}
\frac{\partial u}{\partial \eta}=qu & \text{on}\ \Gamma\\
u = 0 & \text{on}\ \partial \Sigma \setminus \Gamma.
\end{cases}
\]

If $\Sigma$ is an unbounded edged capillary surface, we define the weak (and the strong) index of
$\Sigma$ in the following way: Let $\Omega_1 \subset \Omega_2 \subset \cdots \subset \Omega_n \subset
\cdots \subset \Sigma$ be an exhaustion of $\Sigma$ by bounded open sets. Since the weak index of $\Omega_n$ is a monotone quantity in $n$ we can define
\[ \operatorname{Index}_w(\Sigma) = \lim_{n \to \infty} \operatorname{Index}_w(\Omega_n).\]
Note that this quantity can be $+ \infty$, although this case is outside of the scope of this work.
It follows from standard arguments (see e.g. \cite{Fischer-Colbrie-On-complete-minimal}) that the limit in this definition
is independent of the choice of exhaustion, so the weak index is well-defined.
The definition of the strong index of $\Sigma$ is analogous.

It is clear that 
the weak and strong indices can differ by at most one, hence $\Sigma$ has finite weak index if and only if it has finite strong index. When these quantities are finite, we say that $\Sigma$ has finite index.
In the case where $\Sigma$ is a noncompact capillary surface (has empty edge), the strong index of $\Sigma$ can be
realized as the dimension of a space generated by $L^2$ eigenfunctions defined globally on $\Sigma$. This is the content of
Proposition \ref{prop: L2 characterization}.

\subsection{Harmonic one-forms and vector fields}

On an orientable surface $\Sigma$, let $$\Delta^{[1]}w=(d\delta+\delta d)w$$ be the Hodge Laplacian acting on one-forms. Here $d$ is the exterior differential and $\delta$ is the codifferential. A one-form $w$ is harmonic if $\Delta^{[1]}w=0$. In particular, if a one-form $w$ is closed and coclosed, i.e., $dw=\delta w=0$, then it is harmonic. However, the converse is not true for surfaces  with boundary.

The metric on $\Sigma$ induces a correspondence between one-forms and vector fields, let $\xi = w^\#$ be the vector field corresponding to the one-form $w$. Define the Hodge Laplacian on vector fields as $\Delta_{[1]} \xi=\Delta^{[1]} w$. We say $\xi$ is a harmonic vector field if $\Delta_{[1]}\xi=0$, i.e., the corresponding one-form $w$ is harmonic. Let $\star \xi=(\star w)^{\#}$ where $\star$ is the Hodge operator with respect to the metric on $\Sigma$. If $\operatorname{div}_\Sigma(\xi)=0$ and $\operatorname{div}_\Sigma (\star \xi)=0$, then $\xi$ is harmonic. This is because $\star dw =\operatorname{div}_\Sigma (\star \xi)$ and $\delta w=-\operatorname{div}_\Sigma ( \xi)$. The classical Weitzenbock's formula relates the Hodge Laplacian and rough Laplacian of a vector field, that is 
\begin{equation}\label{eq: Weitzenbock's formula}
    \Delta_{[1]}\xi=\nabla^{*}\nabla \xi+\operatorname{Ric}_\Sigma(\xi),
\end{equation}
where $\nabla^{*}\nabla \xi=-\sum_{i=1}^2\nabla_{e_i}\nabla_{e_i}\xi$ under a local geodesic frame $\{e_1,e_2\}$ and $\operatorname{Ric}_\Sigma(\xi)$ is defined by $\ll \operatorname{Ric}_\Sigma(\xi),X\rl=\operatorname{Ric}_{\Sigma}(\xi,X)$ for any vector field $X\in T\Sigma$.

When $\Sigma$ is compact with boundary, we will consider the following space:
\[\mathcal{H}^1_{T}(\Sigma,\partial\Sigma)=\{\xi\in T\Sigma:\ \operatorname{div}_\Sigma(\xi)=\operatorname{div}_\Sigma(\star\xi)=0 \ \text{on}\ \Sigma \ \text{and}\ \xi \ \text{is tangential along}\ \partial\Sigma \}.\]
It is known that $\operatorname{dim} \mathcal{H}^1_{T}(\Sigma,\partial\Sigma)=2g+r-1$ where $g$ is the genus and $r$ is the number of boundary components of $\Sigma$ (see \cites{AH2021, Pam-freeboundary-inball}).

The following Lemma will be used in the proof of Corollary \ref{particular case in R3}.

\begin{lemma}[\cite{Hodge-decomposition-Schwarz}*{Theorem 3.4.4}]
\label{lemma: max principle harmonic vector field}
Let $\Sigma$ be a complete connected, orientable Riemannian surface with non-empty boundary $\partial\Sigma$. If a harmonic vector field vanishes identically on $U\cap \partial \Sigma\neq \emptyset$ for some open subset $U\subset \Sigma$, then it vanishes identically on $\Sigma$.
\end{lemma}

\section{Index estimates for compact capillary surfaces}\label{sec: index estimates for compact capillary surfaces}
In this section we estimate the weak index of compact capillary surfaces in  $3$-manifolds with boundary. The results in this section generalize those in \cite{AH2021}. According to Lemma 3.1 in \cite{AH2021}, coordinates of $\xi\in \harmonicvfb(\Sigma,\partial \Sigma)$ are admissible functions. Here we will use $\ll\cdot,\cdot\rl$ and $D$ for the Euclidean product and connection respectively.
\begin{proposition}\label{calculation}
Let $M$ be a $3$-dimensional Riemannian manifold isometrically embedded in some Euclidean space $\real^d$. 
Let $\Sigma$ be a compact immersed capillary surface in $M$ at a constant angle $\theta$.
Given a vector field $\xi\in \harmonicvfb(\Sigma,\partial \Sigma)$, denote
\[u_j=\langle \xi,E_j\rangle,\]
where $\{E_j\}_{j=1}^{d}$ is the canonical basis of $\real^d$. Then
\begin{equation}\label{plugin}
\begin{aligned}
\sum_{j=1}^{d}Q(u_j,u_j) = &\int_{\Sigma}\sum_{i=1}^2|\sffM_M(e_i,\xi)|^2+|A_\Sigma(e_i,\xi)|^2- \int_\Sigma \frac{|A_\Sigma|^2+\scalar_{M}+H_\Sigma^2}{2}|\xi|^2 \\
   -&\frac{1}{\sin\theta}\int_{\partial\Sigma} H_{\partial M}|\xi|^2-\cot\theta\int_{\partial\Sigma}H_\Sigma|\xi|^2,
\end{aligned}
\end{equation}
    where $\{e_i\}_{i=1}^2$ is an orthonormal frame on $\Sigma$, $\scalar_M$ is the scalar curvature of $M$, $\sffM_M$ is the second fundamental form of $M$ in $\mathbb{R}^d$ and $H_{\partial M}$ denotes the mean curvature of $\partial M$ with respect to the inner normal vector $-N$.

\end{proposition}

\begin{proof}
Using a local orthonormal basis $\{e_i\}_{i=1}^2$ on $\Sigma$ we have
\[\nabla u_j=\sum_{i=1}^2e_i\langle \xi,E_j\rangle e_i=\sum_{i=1}^2\langle D_{e_i}\xi,E_j\rangle e_i.\]
Since
\[D_{e_i}\xi=\nabla_{e_i}\xi-A_\Sigma(e_i,\xi)\nu+\sffM_M(e_i,\xi),\]
it follows that
\begin{eqnarray}\label{111}
\sum_{j=1}^d|\nabla u_j|^2&=& \sum_{j=1}^d\sum_{i=1}^2|\langle D_{e_i}\xi,E_j\rangle|^2\nonumber\\
&=& \sum_{i=1}^2\sum_{j=1}^d \langle \nabla_{e_i}\xi,E_j\rangle^2+\langle A_\Sigma(e_i,\xi),E_j\rangle^2+\langle \sffM_M(e_i,\xi),E_j\rangle^2\nonumber\\
&=& |\nabla \xi|^2+\sum_{i=1}^2| A_\Sigma(e_i,\xi)|^2+ |\sffM_M(e_i,\xi)|^2.\label{gradient}
\end{eqnarray}
Gauss' equation for $\Sigma$ in $M$ gives us
\[2K_\Sigma=\scalar_M-2\operatorname{Ric}_M(\nu,\nu)-|A_\Sigma|^2+H_\Sigma^2.\]
Hence,
\begin{equation}\label{Ric}\int_\Sigma \operatorname{Ric}_M(\nu,\nu)u_j^2 =\int_\Sigma \left(\frac{\scalar_M}{2}-\frac{|A_\Sigma|^2}{2}+\frac{H_\Sigma^2}{2}-K_\Sigma\right)u_j^2.\end{equation}

It follows from Weitzenbock's formula \eqref{eq: Weitzenbock's formula} and the fact that $\xi$ is harmonic
\[\nabla^*\nabla \xi=-K_\Sigma\xi.\]
By computing the exterior derivative along $\partial \Sigma$ we have
\[d\xi^\flat(\eta,\xi)=\langle \nabla_\eta\xi,\xi\rangle-\langle \nabla_\xi \xi,\eta\rangle.\]
Since $d\xi^\flat=0$,
\[\langle\nabla_\eta\xi,\xi\rangle=\langle\nabla_\xi\xi,\eta\rangle=-k_{\partial\Sigma}|\xi|^2,\]
then
\begin{equation*}
\frac{\partial |\xi|^2}{\partial \eta} = -2\kappa_{\partial \Sigma}|\xi|^2.
\end{equation*}
It follows from the divergence theorem and (\ref{eq: q geodesic curvature}) that
\begin{eqnarray}\label{meancurvature}
\int_\Sigma \Delta|\xi|^2 &=& \int_{\partial \Sigma} \frac{\partial |\xi|^2}{\partial \eta}\nonumber\\
&=& -2\int_{\partial \Sigma}k_{\partial\Sigma}|\xi|^2\nonumber\\
&=& 2\int_{\partial\Sigma}q|\xi|^2-(\frac{1}{\sin\theta}H_{\partial M}+\cot\theta H_\Sigma)|\xi|^2
\end{eqnarray}
which together with 
\[\Delta |\xi|^2=-2\langle\nabla^*\nabla\xi,\xi\rangle+2|\nabla \xi|^2\]
implies that
\begin{eqnarray}\label{boundaryterm}\int_\Sigma|\nabla \xi|^2&=&\int_{\Sigma}\frac{\Delta|\xi|^2}{2}+\int_{\Sigma}\langle \nabla^*\nabla\xi,\xi\rangle\nonumber\\
&=&\int_{\partial \Sigma} q|\xi|^2-(\frac{1}{\sin\theta}H_{\partial M}+\cot\theta H_\Sigma)|\xi|^2-\int_\Sigma K_\Sigma|\xi|^2.\end{eqnarray}
Combining $(\ref{second variation of energy}), (\ref{Ric}), (\ref{boundaryterm})$ and $(\ref{111})$ gives the desired result.
\end{proof}

\begin{theorem}\label{main theorem boundary}
Let $M$ be a $3$-dimensional oriented Riemannian manifold with boundary isometrically embedded in $\real^d$ and let $\Sigma$ be a compact capillary surface immersed in $M$ at a constant angle $\theta$. 
Assume there exist a real number $\beta$ and a $q$-dimensional subspace $\mathbb{W}^q \subseteq \harmonicvfb(\Sigma, \partial \Sigma)$ such that any non-zero $\xi\in \mathbb{W}^q$ satisfies
\begin{equation*}
\begin{aligned}
\int_{\Sigma} & \sum_{i=1}^2|\sffM_M(e_i,\xi)|^2+|\sffM_M(e_i,\star\xi)|^2
-\int_\Sigma (R_{M} +H_\Sigma^2)|\xi|^2\\ &-\int_{\partial\Sigma} (\frac{2}{\sin\theta}H_{\partial M}+2\cot\theta H_\Sigma)|\xi|^2 <2\beta \int_{\Sigma}|\xi|^2.
\end{aligned}
\end{equation*}
Then
\[ \#  \{ \text{eigenvalues\ of\ $\tilde J$\ that\ are\ smaller\ than}\ \beta\} \geq \frac{q-d}{2d}.\]
\end{theorem}
\begin{proof}
Denote $u_j^{\star}=\ll \star\xi,E_j\rl$. Since $|\star\xi|=|\xi|$, then by $(\ref{meancurvature})$,
$$\int_{\Sigma}\Delta|\star\xi|^2=2\int_{\partial\Sigma}q|\xi|^2-(\frac{1}{\sin\theta}H_{\partial M}+\cot\theta H_\Sigma)|\xi|^2.$$
Thus following the proof of Proposition \ref{calculation} gives
\begin{equation}\label{plugin3}
\begin{aligned}
\sum_{j=1}^{d}Q(u^\star_j,u^\star_j) = &\int_{\Sigma}\sum_{i=1}^2|\sffM_M(e_i,\star\xi)|^2+|A_\Sigma(e_i,\star\xi)|^2- \int_\Sigma \frac{|A_\Sigma|^2+\scalar_{M}+H_\Sigma^2}{2}|\xi|^2 \\
   -&\frac{1}{\sin\theta}\int_{\partial\Sigma} H_{\partial M}|\xi|^2-\cot\theta\int_{\partial\Sigma}H_\Sigma|\xi|^2.
\end{aligned}
\end{equation}
Note that whenever $\xi\neq 0$ we may pick $e_1=\frac{\xi}{|\xi|}, e_2=\frac{\star\xi}{|\star\xi|}$ as an orthonormal basis. We have at every point
\begin{equation}\label{keystep1}
\sum_{i=1}^2 |A_\Sigma(e_i,\xi)|^2+|A_\Sigma(e_i,\star\xi)|^2 = \sum_{i,j=1,2}|A_\Sigma(e_i,e_j)|^2|\xi|^2 = |A_\Sigma|^2|\xi|^2.
\end{equation}
Thus summing \eqref{plugin3} and \eqref{plugin} gives 
\begin{equation*}
\begin{aligned}
\sum_{j=1}^{d}Q(u_j,u_j)+Q(u^\star_j,u^\star_j) = &\int_{\Sigma}\sum_{i=1}^2|\sffM_M(e_i,\xi)|^2+|\sffM_M(e_i,\star\xi)|^2\\
 -& \int_{\Sigma}(R_{M} +H_\Sigma^2)|\xi|^2 -\int_{\partial\Sigma} (\frac{2}{\sin\theta}H_{\partial M}+2\cot\theta H_\Sigma)|\xi|^2 .
\end{aligned}
\end{equation*}

Let $k=\#\{ \text{eigenvalues\ of\ $\tilde J$\ that\ are\ smaller\ than}\ \beta\}$ and let $\tilde\phi_1,\ldots,\tilde\phi_k$ the eigenfunctions of $\tilde J$ corresponding to eigenvalues $\tilde\lambda_1\leq\ldots\leq\tilde\lambda_k<\beta$.
Consider the linear map defined by
\begin{eqnarray*}
F: &\mathbb{W}^q\longrightarrow &\mathbb{R}^{2dk+d}\\
& \xi \longmapsto &\left[\int_\Sigma u_j\phi_\alpha, \int_\Sigma u^\star_j\phi_\alpha,\int_{\Sigma} u_j^\star \right],
\end{eqnarray*}
where $\alpha=1,\ldots,k$ and $j=1,\ldots,d$.
By the Rank-Nullity Theorem, if $2dk+d<q$, then there exists a nonzero harmonic tangential vector field $\xi\in \mathcal{H}^1_T(\Sigma,\partial\Sigma)$ such that $u_j,u_j^\star$ are orthogonal to the first $k$ eigenfunctions of $\tilde J$. Moreover, Lemma 3.1 in \cite{AH2021} shows that $\int_\Sigma u_j=0$ for any $j=1,\ldots,d$. Thus
from the Min-max principle for $\tilde J$ it follows that
\[\sum_{j=1}^{d}Q(u_j,u_j)+Q(u^\star_j,u^\star_j)\geq 2\lambda_{k+1}\int_\Sigma|\xi|^2\geq 2\beta\int_\Sigma |\xi|^2\]
which contradicts the assumption in the proposition. Thus $2dk+d\geq q$, or, $k\geq \frac{q-d}{2d}$ as claimed.
\end{proof}

By setting $\beta=0$ in Theorem \ref{main theorem boundary}, we have an index estimate for compact capillary surfaces.
\begin{customthm}{\ref{index estimate free boundary-introduction}}
Let $M$ be a $3$-dimensional oriented Riemannian manifold with boundary isometrically embedded in $\real^d$ and let $\Sigma$ be a compact immersed capillary surface in $M$ at a constant angle $\theta$ with genus $g$ and $r$ boundary components.
Suppose that every non-zero $\xi\in \harmonicvfb(\Sigma,\partial \Sigma)$ satisfies
\begin{equation*}
\begin{aligned}
\int_{\Sigma} & \sum_{i=1}^2|\sffM_M(e_i,\xi)|^2+|\sffM_M(e_i,\star\xi)|^2
-\int_\Sigma R_{M}|\xi|^2\\ &-\int_{\partial\Sigma} \frac{2}{\sin\theta}H_{\partial M}|\xi|^2<\int_\Sigma H_\Sigma^2|\xi|^2+\int_{\partial\Sigma}2\cot\theta H_\Sigma|\xi|^2.
\end{aligned}
\end{equation*}
Then
\[\Index_w(\Sigma)\geq\frac{2g+r-1-d}{2d}.\]
\end{customthm}

\vspace{0.5
cm
}

We now give a proof to Corollary \ref{particular case in R3} in the Introduction. Let $M$ be a domain of $\R^3$ with smooth boundary and assume that $H_{\partial M}+H_\Sigma \cos\theta\geq 0$ along $\partial\Sigma$. If $H_\Sigma=0$ and $H_{\partial M}>0$ at some point of $\partial\Sigma$, then the assumption in Theorem \ref{index estimate free boundary-introduction} is satisfied because $\sffM_M=R_M=0$ and by Lemma \ref{lemma: max principle harmonic vector field}, a nonzero harmonic vector field $\xi$ cannot vanish in a segment of $\partial\Sigma$.  Otherwise, if $H_\Sigma>0$, the assumption in Theorem \ref{index estimate free boundary-introduction} automatically holds.

\section{The structure of noncompact capillary surfaces with finite index}\label{sec: structure of noncompact capillary surfaces with finite index}

In this section we turn to study noncompact capillary surfaces with finite index. We begin it by studying noncompact capillary surfaces in general 3-manifolds with smooth boundary.

\subsection{Noncompact capillary surfaces in general 3-manifolds}
Let $M$ be a three dimensional Riemannian manifold with smooth boundary and let $\Sigma$ be a noncompact capillary 
surface with finite index immersed in $M$. Let $\Omega_1 \subset \Omega_2 \subset \cdots \subset
\Omega_n \subset \cdots $ be an exhaustion of $\Sigma$ by open and bounded sets. Since $\Sigma$ has finite index there must be a $n \in \N$ such that 
$\operatorname{Index}_s(\Omega_m) = \operatorname{Index}_s(\Omega_n)$ for all $m \ge n$. It follows
that $\Sigma \setminus \overline \Omega_n$ is strongly stable. Conversely, it is easy to see that
$\Sigma$ has finite index if it is strongly stable outside some compact set. This fact will be important to
understand noncompact capillary surfaces with finite index.

A Jacobi function in an edged capillary surface $\Sigma$ immersed in $M$ is a function $u \in C^2(\Sigma)$ such that
\[
\begin{cases}
Ju = 0 &\text{ in $\Sigma$}\\
\frac{\partial u}{\partial \eta} - qu = 0 &\text{ on $\Gamma=\partial\Sigma\cap\partial M$}.
\end{cases}
\]
 Jacobi function plays an important role in the study of strongly
stable and finite index capillary surfaces. We will first show that the existence of a positive Jacobi
function characterizes strong stability in the same way as in the theory of minimal surfaces \cite{Fischer-Colbrie-Schoen-The-structure-of-complete-stable}*{Theorem 1}. In the following result we 
denote by $\lambda_1(\Omega, Q)$ the first eigenvalue of the stability operator $Q$ in $\Omega \subset 
\Sigma$.

\begin{proposition}
\label{prop: positive jacobi field}
Let $\Sigma$ be a noncompact capillary surface in a 3-manifold $M$ and let $C \subset \Sigma$ be a compact subset. The followings are equivalent:
\begin{enumerate}
	\item $\lambda_1(\Omega, Q) \ge 0$ for all bounded open sets $\Omega \subset \Sigma \setminus C$.
	\item $\lambda_1(\Omega, Q) > 0$ for all $\Omega$ as above.
	\item There exists a positive Jacobi function $u$ of $\Sigma \setminus C$.
\end{enumerate}
\end{proposition}
\begin{proof}
The arguments showing that $(1) \Rightarrow (2)$ and $(3) \Rightarrow (1)$ are
basically the same as the ones used by Fisher-Colbrie and Schoen 
\cite{Fischer-Colbrie-Schoen-The-structure-of-complete-stable} (see also
\cite{Colding-Minicozzi-A-course-in-minimal-surf}), we omit these calculations here and only show that $(2) \Rightarrow (3)$.

Let $\Gamma = \partial \Sigma \cap \partial M$ and define $H_{0}^1(\Omega)$ to be the closure
of $C_{0}^\infty(\Omega)$ with respect to $H^1(\Omega)$ norm. Note that a weak solution to
\[\begin{cases}
Jv = f_1 &\text{ in $\Omega$}\\
\frac{\partial v}{\partial \eta} - qv = f_2 &\text{ on $\Gamma \cap \Omega$}\\
v = 0 &\text{ on $\partial \Omega \setminus \Gamma$}.
\end{cases}\]
can be characterized as a function $v \in H_{0}^1(\Omega)$ such that for all $\varphi
\in C_{0}^\infty(\Omega)$ we have:
\[Q(v, \varphi) = \int_\Omega f_1\varphi + \int_{\Gamma \cap \Omega} f_2\varphi.\]
The existence of weak solutions for any $f_1 \in L^2(\Omega)$, $f_2 \in
L^2(\Gamma \cap \Omega)$ can be established through the Fredholm alternative using $(2)$ and
standard arguments (see \cite{Evans-PDEs}*{Section 6.2.3}). In particular, there is
a weak solution $v$ to
\[\begin{cases}
Jv = -\text{Ric}_M(\nu,\nu) -|A_\Sigma|^2 &\text{ in $\Omega$}\\
\frac{\partial v}{\partial \eta} - qv = q &\text{ on $\Gamma \cap \Omega$}\\
v = 0 &\text{ on $\partial \Omega \setminus \Gamma$}.
\end{cases}\]
So $u = v + 1$ is a Jacobi function in $\Omega$. Now, using the Harnack inequality
in \cite{Harnack-mixed-conditions} and standard arguments (see \cites{Fischer-Colbrie-Schoen-The-structure-of-complete-stable, Colding-Minicozzi-A-course-in-minimal-surf,
Fischer-Colbrie-On-complete-minimal}) one can show that $u$ is positive.
Then, using the Schauder estimates of Agmon, Douglis
and Nirenberg \cite{Nirenberg-Estimates-near-boundary}*{Theorem 7.3} one can
construct a positive Jacobi function in $\Sigma \setminus C$ by using the same limit
construction as the one used by Fischer-Colbrie and Schoen
\cites{Fischer-Colbrie-Schoen-The-structure-of-complete-stable,
Fischer-Colbrie-On-complete-minimal}.
\end{proof}

Now we are ready to describe the structure of noncompact capillary surfaces under certain curvature
assumptions on $\Sigma$ and $M$. We will show that under these assumptions, noncompact capillary
surfaces also share a key property of minimal surfaces in ambient spaces with nonnegative scalar curvature: They are conformal to compact Riemann surfaces with a finite number of punctures (see \cite{Fischer-Colbrie-On-complete-minimal}*{Theorem 1}).

\begin{customthm}{\ref{thm: Finite topology intro}}
\label{thm: Finite topology}
Let $M$ be an oriented 3-manifold with smooth boundary and let $\Sigma$ be a 
noncompact capillary surface with finite index immersed in
$M$ at a constant angle $\theta$. Assume that 
$R_M + H^2_{\Sigma} \ge 0$ and that one of the following holds:
\[\partial \Sigma \text{ is compact,}\]
or 
\[H_{\partial M} + H_\Sigma\cos \theta  \ge 0 \ \text{along $\partial \Sigma$}.\]
Then $\Sigma$ is conformally equivalent to a compact
Riemann surface with boundary and finitely many points removed, each associated to an end of the surface. Moreover, 
\begin{equation}
\label{eq: finite integral of p}
\int_\Sigma R_M + H_\Sigma^2 + |A_\Sigma|^2 + \int_{\partial \Sigma} H_{\partial M} + H_\Sigma\cos \theta < \infty.
\end{equation}
\end{customthm}
\begin{proof}

Since $\Sigma$ has finite index, there exists a compact set $C \subset  \Sigma$ such that $\Sigma \setminus C$ is strongly stable. For the rest of this proof we will assume that $H_{\partial M} + H_\Sigma\cos \theta  \ge 0$ along $\partial\Sigma$. It is easy to see that the proof in case where $\partial \Sigma$ is compact follows from the same arguments used below by taking $C$ large enough so that $\partial \Sigma \subset C$.

From equation \eqref{eq: q geodesic curvature} we obtain that for all  $\varphi \in C^\infty_0(\Sigma \setminus C)$
\[ \int_\Sigma |\nabla \varphi|^2 - (\text{Ric}_M(\nu, \nu) + |A_\Sigma|^2)\varphi^2 + \int_{\partial \Sigma} \varphi^2\kappa_{\partial \Sigma}  \ge Q(\varphi, \varphi) \ge 0.\]
Applying the same techniques used to prove Proposition \ref{prop: positive jacobi field} we conclude
that there is a positive function $v$ defined on $\Sigma \setminus C$ such that
\[
\begin{cases}
Jv = 0 &\text{ in $\Sigma \setminus C$}\\
\frac{\partial v}{\partial \eta} + v \kappa_{\partial \Sigma}  = 0 &\text{ on $\partial \Sigma \setminus C$}.
\end{cases}
\]
Extend $v$ to all of $\Sigma$ so that it is positive and $\frac{\partial v}{\partial \eta} + 
v\kappa_{\partial \Sigma} = 0$ along all of $\partial \Sigma$. Let $d\tilde s^2 = v^2 ds^2$ where $ds^2$ is the metric of $\Sigma$.
Notice that the metric $d\tilde s$ has Gaussian curvature
\begin{equation}\label{eq:gauss curvature change under conformal change}
    \tilde K_\Sigma = v^{-2}(K_\Sigma - \Delta \log v),
\end{equation}
where $K_\Sigma$ is the Gaussian curvature of $ds^2$. From the Gauss' equation
\[ \Delta v + (|A_\Sigma|^2 + \text{Ric}_M(\nu,\nu))v = 
\Delta v - K_\Sigma v + \frac{1}{2}(R_M + H_\Sigma^2+|A_\Sigma|^2)v = 0,\]
then
\begin{equation}\label{eq: K_Sigma - Delta log u}
K_\Sigma -\Delta \log v = K_\Sigma -\frac{v \Delta v - |\nabla v|^2}{v^2} =\frac{1}{2}(R_M  + H_\Sigma^2+|A_\Sigma|^2)+\frac{|\nabla v|^2}{v^2}\geq 0.
\end{equation}
Hence $\tilde K_\Sigma \ge 0$ on $\Sigma \setminus C$. The geodesic curvature of $\partial \Sigma$ in the 
metric $d\tilde s^2$ is
\begin{equation}\label{eq: geodesic curvature conformal change}
\tilde \kappa_{\partial \Sigma} = v^{-2}\frac{\partial v}{\partial\eta} + v^{-1}\kappa_{\partial \Sigma},
\end{equation}
so $\tilde \kappa_{\partial \Sigma} = 0$.

Now we will show that $(\Sigma, d\tilde s^2 )$ is complete in the sense of metric spaces, that is, every
geodesic of $(\Sigma, d\tilde s^2)$ either exits for infinite time or hits the boundary $\partial \Sigma$. Following
\cite{Fischer-Colbrie-On-complete-minimal} we can construct a divergent geodesic $\gamma$ of
$(\Sigma, d\tilde s^2)$ starting at a point $p \in \partial \Sigma$ that minimizes distance to the boundary of any
disk of $(\Sigma, d\tilde s^2)$ around $p$. To be more precise, $\gamma$ is constructed as the limit of geodesics of
a family of metrics that interpolates between $ds^2$ and $d\tilde s^2$, each of which equals $d\tilde s^2$ in a neighborhood
of $p$ and equals $g$ far away from $p$. It is enough to show that such $\gamma$ has infinite length, 
which can be done by the same arguments used by Fischer-Colbrie \cite{Fischer-Colbrie-On-complete-minimal}.
We conclude that $(\Sigma, d\tilde s^2)$ is complete.

Since $\Sigma$ has smooth boundary we can double it to obtain a smooth surface without boundary $\check \Sigma$. Because $(\Sigma, d\tilde s^2)$ has totally geodesic boundary, the double accepts a metric $d\check{\tilde s}^2$ that restricts to $d\tilde s^2$ in both copies of $\Sigma$ contained in $\check \Sigma$. Note that $d\check{\tilde s}^2$ might not be smooth, but it at least $C^2$ (e.g. Lemma 2.4 in \cite{SYZhang}). We will
argue that $(\check \Sigma, d\check{\tilde s}^2)$ is geodesiclly complete. Let $\Phi : \check \Sigma \to \check \Sigma$ be the
involution coming from the doubling process, and let $\gamma$ be a minimizing geodesic in $(\check \Sigma, 
d\check{\tilde s}^2)$ that is not contained in $\partial \Sigma$. Note that the interior of $\gamma$ can intersect $\partial \Sigma$ at most once, since otherwise we can construct a minimizing geodesic that is not $C^1$ by using
$\Phi$ to reflect the part of $\gamma$ between the first two intersections of $\gamma$ with $\partial \Sigma$ and joining this reflected curve to the rest of $\gamma$.
Now we can use the fact that $(\Sigma, d\tilde s^2)$ is complete to show that if $\gamma$ is divergent it must have infinite
length. 

Since $(\check \Sigma, d\check{\tilde s}^2)$ is complete with nonnegative Gaussian curvature away from a compact set,
it follows from the Huber Theorem that $\check \Sigma$, and hence also $\Sigma$, is finitely connected.
The same arguments used by Fischer-Colbrie \cite{Fischer-Colbrie-On-complete-minimal}*{p. 128}
show that the ends of $(\check \Sigma, d\check{\tilde s}^2)$ are conformal to puncture disks.

Let $d\hat s^2 = u^2ds^2$ where $u > 0$ is such that $u|_{\Sigma \setminus C}$ is a Jacobi function. The same
arguments used to show that $(\Sigma, d\tilde s^2)$ is complete also show that $(\Sigma, d\hat s^2)$ is complete.
Furthermore, \eqref{eq:gauss curvature change under conformal change}--\eqref{eq: geodesic curvature conformal change} apply to this conformal change and imply that in $\Sigma \setminus C$
\[ \hat K_\Sigma \ge \frac{1}{2u^2}(R_M  + H_\Sigma^2+|A_\Sigma|^2),\]
where $\hat K_\Sigma$ is the Gaussian curvature associated to $d\hat s^2$, and in $\partial \Sigma \setminus C$
\[\hat \kappa_{\partial \Sigma} = \frac{1}{u \sin \theta}(H_{\partial M} + H_\Sigma\cos \theta),\]
where $\hat \kappa_{\partial \Sigma}$ is the geodesic curvature of $\partial \Sigma$ associated to $d\hat s^2$.
We then conclude that
\begin{multline} \label{eq: curvature hat g}
 \int_{\Sigma\setminus C}R_M+H_\Sigma^2+|A_\Sigma|^2 dA + 
\int_{\partial \Sigma \setminus C} H_{\partial M} + H_\Sigma\cos \theta d \ell\\
\qquad
\le 2\int_{\Sigma \setminus C} \hat K_\Sigma d\hat A + \sin \theta \int_{\partial \Sigma \setminus C} 
\hat  \kappa_{\partial \Sigma} d \hat \ell,
\end{multline}
where $d A$, $d \hat A$ are the area elements associated to $ds^2$ and $d\hat s^2$, respectively, and $d \ell$, 
$d \hat \ell$ are the length elements associated to $ds^2$ and $d\hat s^2$, respectively.

Since $\hat K_\Sigma$ and $\hat \kappa_{\partial \Sigma}$ are nonnegative outside $C$, we have that
$(\Sigma, d\hat s^2)$ accepts a total curvature and a total geodesic curvature (as defined in \cite{Shiohama-The-geometry-of-total-curvature}). Thus we obtain from the Cohn-Vossen Theorem with boundary
(\cite{Shiohama-The-geometry-of-total-curvature}*{Theorem 2.2.1}) that
\[ \int_{\Sigma \setminus C} \hat K_\Sigma d\hat A + \int_{\partial \Sigma \setminus C} 
\hat  \kappa_{\partial \Sigma} d \hat \ell  < \infty.\]
Combining \eqref{eq: curvature hat g} with the inequality above completes the proof of the theorem.
\end{proof}

\begin{remark}\label{rmk: description of conformal structure}
One can think of a capillary surface $\Sigma$ as in the statement of Theorem \ref{thm: Finite topology} as a compact Riemann surface $\overline \Sigma$ with boundary and with points $p_1, \ldots, p_k$ removed. We can order these points so that $p_1, \ldots, p_\ell$ are in the interior of $\overline \Sigma$ (there might be no interior punctures, in which case we take $\ell =0$) and $p_{\ell + 1}, \ldots, p_{k}$ are in the boundary of $\overline \Sigma$. Then each $p_i$ for $i \in \{ 1, \ldots, \ell\}$ corresponds to an end of $\Sigma$ that is away from its boundary, we will refer to these ends as \textbf{interior ends}. For $i \in \{\ell + 1, \ldots, k\}$, each $p_i$ corresponds to an end of $\Sigma$ that contains a noncompact part of $\partial \Sigma$, we refer to these ends as \textbf{boundary ends}.
\end{remark}
\begin{remark}\label{rmk: description of ends}
It is clear that each interior end of $\Sigma$ is conformal to a punctured disk $D_* = \{z \in \C: 0 < |z| < 1 \}$. The boundary ends of $\Sigma$ are conformal to a semi-open punctured half-disk $D_*^+ = \{ z \in D_* : \textup{Re}(z) \ge 0\}$. To see this fact note that a neighborhood of a point $p_i$ for $i \in \{ \ell + 1, \ldots, k\}$ in the double $\check \Sigma$ is conformal to a punctured disk. We can assume that this neighborhood is symmetric with respect to the isometry $\Phi$, and hence $\Phi$ induces a conformal, orientation-reversing diffeomorphism on the punctured disk. The composition of this transformation with complex conjugation must be an injective holomorphic map from the punctured disk to itself, and hence must be a rotation. So we can assume that the part of this punctured disk corresponding to $\Sigma$ is $D_*^{+}$.
\end{remark}

In what follows we show two consequences of Theorem \ref{thm: Finite topology}. The first one is:
\begin{corollary}\label{corollary: capillary must be minimal}\label{corollary: surface must be minimal}
Let $M$ be a weakly mean-convex domain of $\mathbb{R}^3$ and let $\Sigma$ be a noncompact capillary surface with finite index immersed in $M$ at a constant angle $\theta\in(0,\pi/2)$.
Then $\Sigma$ must be a minimal surface. If $M$ is a half-space, the assumption on the angle can be weakened to $\theta\in (0,\pi/2].$

\end{corollary}

\begin{proof}
It follows from Theorem \ref{thm: Finite topology} that $\Sigma$ is conformally equivalent to a compact Riemann surface $\bar{\Sigma}$ punctured at finitely many points, and
\begin{equation}\label{eq: integral of mean curvature is finite}
\int_\Sigma H_\Sigma^2+\int_{\partial\Sigma} H_{\Sigma}\cos\theta<\infty.
\end{equation}
Suppose $\Sigma$ has an interior end. We choose a geodesic ray $\gamma:[0,\infty)\rightarrow \Sigma$ contained in an interior end.
Fix small $\epsilon_0>0$ and consider a sequence of points $p_j=\gamma(2j\epsilon_0)$ such that the geodesic disks $D_{\epsilon_0}(p_j)$ of radius $\epsilon_0$ around $p_j$  satisfy $D_{\epsilon_0}(p_j)\cap D_{\epsilon_0}(p_k)=\emptyset$ whenever $j\neq k$. From the uniform lower bound on the area of geodesic disks in \cite{Frensel-CMCsurfaces-infinite-volume}*{Theorem 3} it follows that
\[A(\Sigma)\geq \sum_{j=0}^n A(D_{\epsilon_0}(p_j))\geq (n+1)C(\epsilon_0), \]
where $A$ is the area and $C(\epsilon_0)$ is a positive constant depending only on $\epsilon_0$. Since above inequality holds for any $n \in \N$, we obtain that $\Sigma$ has infinite area, thus it follows from (\ref{eq: integral of mean curvature is finite}) that $H_\Sigma=0$. 

If $\Sigma$ has no interior end, it has at least one boundary end since it is noncompact. Then it is easy to see that $\partial\Sigma$ has infinite length. Thus, if $\theta\in (0,\pi/2)$, we must have $H_\Sigma=0.$ This completes the first part of the Theorem. 

If $\theta\in (0,\pi/2]$ and $M$ is a half-space, we only need to deal with the special case when $\theta=\pi/2$ and $\Sigma$ has no interior ends. Denote a noncompact boundary component of $\Sigma$ by $E$, choose a sequence of disjoint points $p_j\in E$ and a small $\epsilon_0>0$ such that geodesic disks $D_{\epsilon_0}(p_j)$ are disjoint. By using a similar idea to \cite{Frensel-CMCsurfaces-infinite-volume}, we claim that $A(D_{\epsilon_0}(p_j))\geq C(\epsilon_0)$ where $C(\epsilon_0)$ is a positive constant depending only on $\epsilon_0$. Indeed, without lost of generality, we assume that $p_j$ is the origin and let $D_\mu$ be the geodesic disk centered at origin with radius $\mu$. For small $\mu$, denote $\Gamma_\mu=\partial D_\mu\cap P$ and $I_\mu=\partial D_\mu\setminus \Gamma_\mu$ where $P$ is the plane, i.e., boundary of the half-space.  It is well known that $\Delta_\Sigma |x|^2=4+2H_\Sigma\ll \nu,x\rl$. We integrate both sides of this equation on the disk $D_\mu$ to obtain
\[(4-2H_\Sigma\mu)A(D_\mu)\leq 4|D_\mu|-2H_\Sigma\int_{D_\mu}|x|\leq \int_{D_\mu}\Delta_\Sigma |x|^2.\]
By the divergence theorem, 
\[\int_{D_\mu}\Delta_\Sigma |x|^2= \int_{\partial D_\mu}2\ll \eta,x\rl\leq 2\mu l(I_\mu),\]
where $l$ denotes the length. In addition, from the co-area formula it follows that
\[\frac{\partial A(D_\mu)}{\partial \mu}\geq l(I_\mu).\]
Combining above equations results in 
\[\left(\frac{A(D_\mu)}{\mu^2}\right)'\geq -H_\Sigma\frac{A(D_\mu)}{\mu^2},\]
equivalently,
\[\left(\text{log} \frac{A(D_\mu)}{\mu^2}\right)'\geq -H_\Sigma.\]
We then show the claim by integrating both sides over $(a,\epsilon_0)$ and letting $a$ tend to zero. Hence the surface has infinite area and thus (\ref{eq: integral of mean curvature is finite}) implies again that $H_\Sigma=0.$
\end{proof}

The second corollary is stated as follows.
\begin{corollary}\label{cor: discription strongly stable}
Let $M$ be an oriented 3-manifold with smooth boundary and let $\Sigma$ be a strongly stable noncompact capillary
surface immersed in $M$ at a constant angle $\theta$. Assume that $R_M + H^2_\Sigma \ge 0$ in $\Sigma$ and 
$H_{\partial M} + H_\Sigma \cos \theta \ge 0$ along $\partial \Sigma$. Then the compact Riemann surface
$\overline \Sigma$ is a disk and the ends of $\Sigma$ can only have one of the following configurations:
\begin{enumerate}
    \item There are two boundary ends and no interior ends.
    \item There are no boundary ends and a single interior end.
    \item There is a single boundary end and no interior ends.
\end{enumerate}
Moreover, if (1) or (2) holds, then $\Sigma$ is totally geodesic, $R_M = 0$ in $\Sigma$ and $H_{\partial M} = 0$
along $\partial \Sigma$.
\end{corollary}
\begin{proof}
Since $\Sigma$ is strongly stable, it accepts a globally defined positive Jacobi function $u$. Considering a
metric $d\hat s^2$ as in the proof of Theorem \ref{thm: Finite topology} we can conclude that
$\hat K_\Sigma \ge 0$ on $\Sigma$ and $\hat \kappa_{\partial \Sigma} \ge 0$ in all of $\partial \Sigma$,
furthermore
\[\int_{\Sigma} \hat K_\Sigma d \hat A \ge 0\]
with equality if and only if $\Sigma$ is totally geodesic and $R_M = 0$ on $\Sigma$, and
\[\int_{\partial \Sigma} \hat \kappa_{\partial\Sigma} d \hat \ell \ge 0\]
with equality if and only if $H_{\partial M} + H_\Sigma \cos \theta = 0$ along $\partial \Sigma$.

It follows from the Cohn-Vossen Theorem \cite{Shiohama-The-geometry-of-total-curvature}*{Theorem 2.2.1} that
\[\int_{\Sigma} \hat K_\Sigma d \hat A + \int_{\partial \Sigma} \hat \kappa_{\partial\Sigma} d \hat \ell 
\le \pi(4 - 4g - 2r - k - \ell),\]
where $g$ and $r$ are the genus and number of boundary components of $\overline{\Sigma}$, respectively, and $k$,
$\ell$ are the total number of ends and the number of interior ends of $\Sigma$, respectively. Since $\Sigma$ has a boundary and is noncompact, we conclude that $k \ge 1$ and $r \ge 1$. It then follows that $g = 0$, $r = 1$, and $k + \ell = 1$ or $2$.
Note that 
\[k + \ell = \#\{\text{boundary ends}\} + 2 \times \#\{\text{interior ends}\},\]
so the only possible configurations of ends are the ones listed in (1)--(3).
In the case where $k + \ell = 2$, we conclude that $\hat K_\Sigma$ and $\hat \kappa_{\partial \Sigma}$ vanish,
hence $\Sigma$ is totally geodesic, $R_M = 0$ along $\Sigma$ and $H_{\partial M} = 0$ along $\partial \Sigma$, proving the rigidity statement.
\end{proof}

\begin{remark}
Observe that each of the three possibilities of Corollary \ref{cor: discription strongly stable} does occur. For an example of:
\begin{enumerate}
    \item consider $M=\mathbb{R}^2\times [0,1]$ and let $\Sigma$ be an infinite flat strip in $M$ meeting the boundary at a constant angle $\theta\in(0,\pi)$;
    \item let $M = \mathbb{S}^1\times \mathbb{R}^+\times \mathbb{R}$ and take $\Sigma = \mathbb{S}^1\times \mathbb{R}^+ \times\{0\}$;
    \item consider a half-plane in a half-space of $\mathbb{R}^3$.
\end{enumerate}
\end{remark}

\begin{remark}
Combining the inequality in Theorem \ref{thm: Finite topology intro} and Corollary \ref{cor: discription strongly stable}, we can also show that there are no strongly stable noncompact capillary minimal surfaces $\Sigma$ in a 3-manifold $M$ with nonnegative scalar curvature and uniformly mean-convex boundary, i.e., $R_M\geq 0$ and $H_{\partial M}\geq c>0.$ Indeed, with such curvature assumptions the inequality in Theorem \ref{thm: Finite topology intro} implies that $\partial \Sigma$ must be compact. However, the arguments in the proof of Corollary \ref{cor: discription strongly stable} show that $\Sigma$ has a single boundary end and no interior ends, which is a contradiction.
\end{remark}

\subsection{Weakly stable capillary surfaces in a half-space of $\R^3$}

Capillary surfaces in a half-space of $\R^3$ are analogous to
complete minimal and CMC surfaces in $\R^3$ in the sense that these are local models for capillary
surfaces in any 3-manifold.
In the following we  characterize strongly stable capillary
surfaces in a half-space of $\R^3$.

\begin{customthm}{\ref{maintheorem1}}
Let $\Sigma$ be a noncompact capillary surface immersed in a half-space of $\R^3$ at a constant angle $\theta$. Assume that $H_\Sigma\cos\theta\geq0$. Then $\Sigma$ is weakly stable if and only if it is a half-plane.
\end{customthm}
\begin{proof}
It is clear from \eqref{second variation of energy} and \eqref{eq: q geodesic curvature} that half-planes are weakly stable in a half-space of $\R^3$. To show the other implication,
let us assume without loss of generality that the half-space is $\{x_3 \ge 0\}$. It suffices to prove the following claim.
\begin{claim}
If $|A_\Sigma|$ is not identically zero, then there exists a compactly supported function $v$ on $\Sigma$ such that
\[Q(v,v)<0, \ \ \text{and}\ \ \ \int_\Sigma v=0.\]
\end{claim}

Since $\Sigma$ is conformally equivalent to a compact Riemann surface $\bar{\Sigma}$ with boundary and finite punctures $p_1,\ldots,p_n$. Without lost of generality, we assume that $p_1,\ldots,p_k$ are punctures in the interior while $p_{k+1},\ldots,p_n$ are punctures on the boundary. Thus there exists a compact subset $\Sigma_0\subset \Sigma$ such that each component $\Sigma_i$ of $\Sigma\setminus \Sigma_0$ is conformally equivalent to either a semi-cylinder $\mathbb{S}^1\times (0,\infty)$ or a half-semi-cylinder $\mathbb{S}^1_{+}\times (0,\infty)$ with the standard product metric.

Since $|A_\Sigma|$ is not identically zero, there exists $p\in\Sigma$ such that $|A_\Sigma(p)|\neq 0$. We can choose $\Sigma_0$ such that $p\in\Sigma_0$ and 
\begin{equation}\label{choose r}
    \int_{\Sigma_0}|A_\Sigma|^2>\frac{12(n+k)\pi}{r(1-\cos\theta)}
\end{equation}
for some fixed $r>0$ large enough that we will choose later.

We parametrize $\Sigma_i$ for $i=1,\ldots,n$ by conformal coordiantes $(\theta,y_i)\in \mathbb{S}^1\times (0,\infty).$ For each $a>0$, we can define functions $\varphi_i:\Sigma_i\rightarrow \mathbb{R}_+$ by
\begin{equation}
\varphi_i(y_i)=
\begin{cases}
1-\frac{y_i}{r}, & 0\leq y_i\leq 2r\\
-1,& 2r\leq y_i\leq 2r+a\\
-1+\frac{y_i-(2r+a)}{r},& 2r+a\leq y_i\leq 3r+a\\
0, & 3r+a\leq y_i.
\end{cases}
\end{equation}
We then define a function $\varphi_a:\Sigma\rightarrow \mathbb{R}_+$ such that $\varphi_a\equiv 1$ on $\Sigma_0$ and $\varphi_a=\varphi_i$ on $\Sigma_i$ for $i=1,\ldots,n$. Let $u:\Sigma\rightarrow \mathbb{R}_+$ be
\[u=\frac{1}{\sin\theta}+\cot\theta\ll \nu,-E_3\rl.\]
where $\nu$ is the unit normal vector field of $\Sigma$. Note that $0<\frac{1-\cos\theta}{\sin\theta}\leq u\leq \frac{2}{\sin\theta}.$ We claim that there exists a $a=a_0>0$ such that
\[\int_\Sigma u\varphi_{a_0}<0.\]
In fact, 
\begin{align*}
    \int_\Sigma u\varphi_a&\leq \int_{\Sigma_0\cup \left(\cup_{i=1}^n\{\ 0\leq y_i\leq 2r\}\right)}u\varphi_a+\int_{\cup_{i=1}^n\{\ 2r \leq y_i\leq 3r+a\}}-u\\
    &\leq \int_{\Sigma_0\cup \left(\cup_{i=1}^n\{\ 0\leq y_i\leq 2r\}\right)}u\varphi_a-\frac{1-\cos\theta}{\sin\theta}\sum_{i=1}^n|\{\ 2r \leq y_i\leq 3r+a\}|
\end{align*}
where the first term on the right-hand side is independent of $a$, thus a fixed number. Since the area of $\Sigma$ is infinite, we can choose $a_0>0$ such that the desired inequality holds.

We now define new functions $\psi_i:M_i\rightarrow \mathbb{R}_+$ by
\begin{equation}
\psi_i(y_i)=
\begin{cases}
1-\frac{y_i}{r}, & 0\leq y_i\leq r+b\\
-b/r,& r+b\leq y_i\leq 3r+a_0-b\\
-1+\frac{y_i-(2r+a_0)}{r},& 3r+a_0-b\leq y_i\leq 3r+a_0\\
0, & 3r+a_0\leq y_i.
\end{cases}
\end{equation}
We then define a function $\psi_{a_0,b}:\Sigma\rightarrow \mathbb{R}_+$ such that $\psi_{a_0,b}=1$ on $\Sigma_0$ and $\psi_{a_0,b}=\psi_i$ on $\Sigma_i$. Note that $\psi_{a_0,r}=\varphi_{a_0}$ and $\psi_{a_0,0}\geq 0.$ Since the integral $\int_\Sigma u\psi_{a_0,b}$ is continuous in $b$ and $\int_\Sigma u\psi_{a_0,0}>0$, there exists a $0<b_0\leq r$ such that
\[\int_{\Sigma}u\psi_{a_0,b_0}=0.\]

Let $v=u\psi_{a_0,b_0}$, notice it is a piece-wise smooth function with compact support. We now calculate $Q(v,v)$. First note that
\begin{align*}
u|_{\partial \Sigma} 
 &= \frac{1}{\sin\theta} + \cot \theta \langle
 \sin\theta T + \cos \theta E_3, - E_3 \rangle\\
 &= \frac{1}{\sin\theta} - \frac{\cos^2\theta}{\sin \theta}\\
 &= \sin \theta\\
 &= \langle \eta, -E_3 \rangle.
\end{align*}
It follows from simple computations that $\eta$ is a principal direction of $\Sigma$ (see e.g. \cite{Ainouz-Souam-capillarysurfaceinslab}*{Lemma 2.2}, \cite{Wang-Xia-Uniqueness-of-stable}*{Prop 2.1}). Using this fact we obtain
\begin{align}
\frac{\partial u}{\partial \eta} 
&= \cot \theta \langle D_\eta \nu, -E_3 \rangle \nonumber \\
&= \cot \theta A_\Sigma(\eta,\eta) \langle \eta, -E_3 \rangle \nonumber\\
&= \cot \theta A_\Sigma(\eta,\eta)u|_{\partial \Sigma}. \nonumber\\
&= q u. \label{eq: outer normal varphi}
\end{align}
Since $\Delta \nu = -|A_\Sigma|^2\nu$ in the coordinate-wise sense (see e.g \cite{doCarmo-Stability-of-hypersurfaces}*{Proposition 2.24}), we have
\begin{equation}\label{eq: lap varphi}
\Delta u + |A_\Sigma|^2u = \frac{|A_\Sigma|^2}{\sin \theta}.
\end{equation}

Then we have that 
\begin{align}\label{calculate Q}
    Q(v,v)& = \int_\Sigma u^2 |\nabla \psi_{a_0,b_0}|^2 + \psi_{a_0,b_0}^2|\nabla u|^2 + 2u\psi_{a_0,b_0}\langle \nabla u, \nabla \psi_{a_0,b_0} \rangle - \int_{\partial \Sigma} qu^2\psi_{a_0,b_0}^2 \nonumber\\
    &= \int_\Sigma u^2 |\nabla \psi_{a_0,b_0}|^2 + \frac{1}{2}\psi_{a_0,b_0}^2\Delta(u^2) - \psi_{a_0,b_0}^2u\Delta u + \frac{1}{2}\langle \nabla(u^2), \nabla(\psi_{a_0,b_0}^2) \rangle\nonumber\\
    & \qquad - \int_{\partial \Sigma} qu^2\psi_{a_0,b_0}^2 \nonumber \\
    &=  \int_\Sigma u^2 |\nabla \psi_{a_0,b_0}|^2 - \psi_{a_0,b_0}^2u\Delta u +  \int_{\partial \Sigma} u\psi_{a_0,b_0}^2( \frac{\partial u}{\partial \eta} - qu) \nonumber \\
    &= \int_\Sigma u^2|\nabla \psi_{a_0,b_0}|^2-\psi_{a_0,b_0}^2u\frac{|A_\Sigma|^2}{\sin\theta}\nonumber\\
    &\leq \frac{4}{\sin^2\theta}\int_\Sigma |\nabla \psi_{a_0,b_0}|^2-\frac{1-\cos\theta}{\sin^2\theta}\int_\Sigma \psi_{a_0,b_0}^2|A_\Sigma|^2\nonumber\\
    &\leq \frac{4}{\sin^2\theta}\sum_{i=1}^n\int_{\Sigma_i}|\nabla \psi_{a_0,b_0}|^2-\frac{1-\cos\theta}{\sin^2\theta}\int_{\Sigma_0} |A_\Sigma|^2
\end{align}
Since the Dirichlet energy is conformally invariant in dimension two, we have that for $i=1,\ldots,k$
\[\int_{\Sigma_i}|\nabla \psi_{a_0,b_0}|^2=\int_{\Sigma_i}|D\phi_i|^2\ d\theta dy_i=\frac{2\pi(r+2b_0)}{r^2}\leq \frac{6\pi}{r}\]
and for $i=k+1,\ldots,n$
\[\int_{\Sigma_i}|\nabla \psi_{a_0,b_0}|^2\leq \frac{3\pi}{r}.\]
Then it follows from $(\ref{calculate Q})$ and $(\ref{choose r})$ that
\[Q(v,v)<0.\]
This completes the proof of the claim. Since the surface $\Sigma$ is weakly stable, by the claim we must have that $|A_\Sigma|\equiv 0$ and thus $\Sigma$ is a half-plane.
\end{proof}

\begin{remark}\label{remark: regidity compact case}
Note that same arguments show that a weakly stable noncompact capillary surface (without angle condition) immersed in a half-space of $\R^3$ cannot have compact boundary. Also, using the same test function we constructed above, one can show directly that there is no strongly stable compact capillary surface immersed in a half-space of $\R^3$.
\end{remark}

\subsection{$L^2$ characterization of the strong index}
Fisher-Colbrie \cite{Fischer-Colbrie-On-complete-minimal}*{Proposition 2} showed that the index of a complete noncompact minimal surface can be realized as the cardinality of a set of $L^2$ eigenfunctions defined globally on the surface. This fact is important to study the relation between the index and the topology of these surface (see e.g. \cite{Ros-onesided-minimalsurface}*{Theorem 17}).
In what follows we will show that something similar also holds for capillary surfaces.
Note that the functions $f_i$ below are eigenfunctions in the weak sense.

\begin{proposition}\label{prop: L2 characterization} 
Let $\Sigma$ be a noncompact capillary surface immersed in a 3-manifold with smooth boundary
and let $I = \Index_s(\Sigma)$. Assume $I < \infty$,
then there exists a subspace $W$ of $L^2(\Sigma)$
having an orthonormal basis $f_1, \ldots, f_I$ consisting of
eigenfunctions for $Q$ with eigenvalues $\lambda_1, \ldots, \lambda_I$,
respectively, such that each $\lambda_i < 0$. Moreover, for all $\varphi \in
C^\infty_0(\Sigma) \cap W^\perp$ we have
\[Q(\varphi, \varphi) \ge 0.\]

\end{proposition}
\begin{proof}
For all $R >0$, let $B_{R} \subset \Sigma$ be a ball of 
radius $R$ centred at
some point that will be fixed through this proof.
Take $R_0 \ge 4$ such that $\Sigma \setminus \overline{B_{R_0}}$ is strongly stable. For $R \ge R_0$, let $\zeta \in
C^\infty(\Sigma)$ be such that
\begin{align*}
	\zeta &= 0 \text{ on $B_R$},\\
	\zeta &= 1 \text{ on $\Sigma \setminus B_{2R}$},
\end{align*}
and
\[|\nabla \zeta| \le \frac{6}{R},\ |\nabla \zeta|^2 < \frac{4(1 - \zeta^2)}{R^2}.\]
For the construction of $\zeta$ see \cite{Fischer-Colbrie-On-complete-minimal}.
For simplicity we will write $\text{Ric}_M(\nu, \nu) + |A_\Sigma|^2$ as $p$.
Since $\Sigma \setminus \overline{B_R}$ is strongly stable, we have that for all $\varphi
\in C^\infty_0(\Sigma)$
\begin{equation}\label{eq: stability for xi phi}
\int_{\partial \Sigma} q(\zeta \varphi)^2 + \int_\Sigma p (\zeta\varphi)^2
\le \int_\Sigma |\nabla \zeta \varphi|^2
= \int_\Sigma \zeta^2|\nabla \varphi|^2 + 2\zeta\varphi\ll\nabla \zeta, \nabla
\varphi\rl + \varphi^2|\nabla \zeta|^2.
\end{equation}
It follows from the Cauchy-Schwarz inequality and the arithmetic-geometric mean inequality that
\[\int_\Sigma 2\zeta\varphi\ll\nabla \zeta, \nabla \varphi\rl \le 
\int_{B_{2R}} \zeta^2 \varphi^2 + |\nabla \zeta|^2|\nabla \varphi|^2.\]
Adding $Q(\varphi, \varphi)$ to both sides of \eqref{eq: stability for xi phi}
and rearranging gives that
\begin{align}
-\int_{\partial \Sigma} (1 - \zeta^2) q\varphi^2 -
&\int_\Sigma ( 1 - \zeta^2)p \varphi^2 + 
\int_\Sigma ( 1 - \zeta^2)|\nabla \varphi|^2 \label{eq: stability eq 2}\\
& \le Q(\varphi, \varphi) + \int_{B_{2R}} \varphi^2 +
\frac{4}{R^2}\int_\Sigma (1 - \zeta^2)|\nabla \varphi|^2.\nonumber
\end{align}
Using Lemma \ref{lemma: poincare} we conclude
\begin{equation}\label{eq: bound for int phi}
\int_{\partial \Sigma} (1 - \zeta^2)q\varphi^2 
\le \sup_{B_{2R}}|q| \int_{\partial \Sigma} (1 - \zeta^2)\varphi^2\le
\frac{1}{4} \int_\Sigma (1 - \zeta^2)|\nabla \varphi|^2 + C_R \int_{B_{2R}}
\varphi^2
\end{equation}
for some constant $C_R$ depending on $R$. Since $R \ge 4$ it follows from \eqref{eq: stability eq 2} and \eqref{eq: bound for int phi} that
\begin{equation}\label{eq: bound for nabla phi}
\int_{B_R}|\nabla \varphi|^2 \le 2Q(\varphi, \varphi) +
C_R'\int_{B_{2R}}\varphi^2
\end{equation}
for some new constant $C_R'$ depending only on $R$.

Let $\rho > R_0$, then there are functions $f_{1, \rho},
\ldots, f_{I, \rho} \in L^2(B_\rho)$ forming an orthonormal 
set such that each $f_{i, \rho}$ is an eigenfunction of $Q$ in the sense
that
\[
\begin{cases}
\Delta f_{i, \rho} + pf_{i, \rho} + \lambda_{i, \rho} f_{i, \rho} = 0
&\text{ in } B_{\rho}\\
\frac{\partial f_{i, \rho}}{\partial \eta} -qf_{i, \rho} = 0 
&\text{ on } \partial \Sigma \cap B_\rho\\
f_{i, \rho} = 0 &\text{ on } \partial B_\rho \setminus \partial \Sigma.
\end{cases}
\]
for some eigenvalue $\lambda_{i, \rho} < 0$. Extend $f_{i, \rho}$  to be 0
outside $B_\rho$. Since $\max_i \lambda_{i, \rho}$ is decreasing in
$\rho$ we can assume that $\lambda_{i,\rho} \le -\epsilon_0$ for some constant 
$\epsilon_0 > 0$. In addition, from \eqref{eq: bound for nabla phi} it follows that
$Q(f_{i, \rho}, f_{i, \rho}) \ge -C_R  \int_\Sigma f_{i, \rho}^2$ , thus there is a
constant $C > 0$ such that $\lambda_{i, \rho} \ge - C$ for all $i$ and $\rho$ as
above. Notice that
\begin{align*}
\int_\Sigma 2 \zeta f_{i, \rho} \ll \nabla f_{i, \rho}, \nabla \zeta \rl
&= \frac{1}{2} \int_\Sigma \ll \nabla f_{i, \rho}^2 ,\nabla \zeta^2 \rl\\
&= \int_{\partial \Sigma} \zeta^2 f_{i, \rho} \frac{\partial f_{i, \rho}}{\partial
\eta} - \int_\Sigma \zeta^2( f_{i, \rho}\Delta f_{i, \rho} + |\nabla f_{i, \rho}|^2)\\
&= \int_{\partial \Sigma} q(\zeta f_{i, \rho})^2  -
\int_\Sigma \zeta^2( -(p + \lambda_{i, \rho})f_{i, \rho}^2 + |\nabla f_{i, \rho}|^2).
\end{align*}
Substituting it in \eqref{eq: stability for xi phi} we can conclude that
\[ - \lambda_{i, \rho}\int_\Sigma \zeta^2 f_{i, \rho}^2 \le  \int_\Sigma |\nabla
\zeta|^2 f_{i, \rho}^2 \le \frac{36}{R^2}.\]
Since $\lambda_{i, \rho} \le -\epsilon_0$ we can conclude that
\begin{equation}\label{eq: no mass loss at infty}
\int_{\Sigma \setminus B_{2R}} f_{i, \rho}^2 \le \frac{C'}{R^2}
\end{equation}
for some constant $C'$. Inequality \eqref{eq: bound for nabla phi} and the fact that $Q(f_{i, \rho}, f_{i, \rho}) <0$ imply
\begin{equation}\label{eq: local energy bound}
\int_{B_R} |\nabla f_{i, \rho}|^2 + f_{i, \rho}^2 \le C_R' + 1.
\end{equation}
By the Reillich-Kondrachov compactness theorem the embedding $H^1(B_R) \to L^2(B_R)$ and the trace map
$H^1(B_R) \to L^2(B_R \cap \partial \Sigma)$ are both compact (see \cite{Adams-Sobolev-spaces}*{Theorem 6.3}).
So using \eqref{eq: local energy bound}, a standard diagonal argument can be used to construct a sequence 
$R_j \to \infty$ and functions $f_1, \ldots, f_I$ in $\Sigma$ such that for all
$i = 1, \ldots, I$ and $R > 0$ we have as $j \to \infty$
\begin{gather*}
f_{i, R_j} \rightarrow f_i \text{ strongly in } L^2(B_{R}),\\
f_{i, R_j}|_{\partial \Sigma} \rightarrow f_i|_{\partial \Sigma} 
\text{ strongly in } L^2(B_{R}\cap \partial \Sigma),\\
f_{i, R_j} \rightarrow f_i \text{ weakly in } H^1(B_{R}).
\end{gather*}
It follows that the functions $f_i$ are eigenfunctions of $Q$ in the weak sense, and the corresponding eigenvalues are
$\lambda_i = \lim_{j \to \infty} \lambda_{i, R_j}$. Furthermore, it follows from \eqref{eq: no mass loss at
infty} that $f_1, \ldots, f_I$ form an orthonormal set. Let $W$ be the linear space spanned by these
functions and take $\varphi \in C^\infty_0(\Sigma) \cap W^\perp$, then for all
$\rho > R$ we can write
\[\varphi = \sum_{i = 1}^I  a_{i, \rho} f_{i, \rho} + \varphi_\rho \]
where $a_{i, \rho} = \int_\Sigma \varphi f_{i, \rho}$ and $\varphi_\rho \perp f_{i,
\rho}$. It is not hard to see that $a_{i, R_j} \to 0$ as $j \to \infty$. Since
$\lambda_{i, R_j}$ are uniformly bounded below, we can conclude that
\[Q(\varphi, \varphi) = Q(\varphi_{R_j}, \varphi_{R_j}) + \epsilon_{R_j}\ge
\epsilon_{R_j} \to 0.\]
Hence $Q(\varphi, \varphi) \ge 0$. This completes the proof.
\end{proof}

In the following corollary we give an application of Proposition \ref{prop: L2 characterization}. It further characterizes the topology of an \textit{index one} free boundary CMC surface in a 3-manifold with boundary.

\begin{corollary}
Let $M$ be an oriented flat 3-manifold with smooth weakly convex boundary and let $\Sigma$ be a 
noncompact free boundary CMC surface with strong index one immersed in
$M$.
Then $\Sigma$ is conformally equivalent to a compact
Riemann surface  $\bar{\Sigma}$ of genus $g$ and $r$ boundary components punctured at $k$ points, and
\[k+\ell\leq 4\left\lfloor\frac{1-g}{2}\right\rfloor - 2r + 8,\]
where $\ell$ is the number of boundary punctures. In particular, $g\leq 3$, $r \le 3$, $g + r \le 4$ and $k + \ell \le 6$.
\end{corollary}
\begin{proof}
By Theorem \ref{thm: Finite topology intro}, $\Sigma$ is conformally equivalent to a compact Riemann surface $\bar{\Sigma}$ of genus $g$ with $r$ boundary components punctured at $k$ points $p_1,\ldots,p_k$. Let $ds^2$ be the metric on $\Sigma$. Let $\rho$ be a positive smooth function such that $(\Sigma,\rho ds^2)$ is complete and has finite area. Indeed, since each end is conformal either to a punctured disk or  half-disk, we can take conformal coordinates centered at zero near each end and we can let $\rho=1/|z|$ in each of these neighborhoods,  so that the metric is complete and the surface has finite area. Let $d\tilde{s}^2=\rho ds^2$. Since the Dirichlet integral is conformally invariant in dimension 2, we have that 
\[\tilde{Q}(u,u)=\int_{\Sigma}|\tilde{\nabla}u|^2-|A_\Sigma|^2\rho^{-1}u^2 d\tilde A -\int_{\partial\Sigma}\rho^{-1/2}qu^2 d \tilde \ell\]
coincides with the stability operator $Q(u,u)$ for any compactly supported function $u$ on $\Sigma$. Since $(\Sigma, ds^2)$ has index one, $\tilde{Q}$ has index one. By Proposition \ref{prop: L2 characterization}, there exists an eigenfunction $f_1\in L^2(\Sigma,d\tilde{s}^2)$ corresponding to a negative eigenvalue for the following eigenvalue problem:
\begin{equation*}
\begin{cases}
\tilde{\Delta} u+|A_\Sigma|^2\rho^{-1}u+\lambda u=0 & \text{in}\ \Sigma\\
\frac{\partial u}{\partial \tilde{\eta}}=\rho^{-1/2}qu & \text{on}\ \partial \Sigma.
\end{cases}
\end{equation*}
The construction in Proposition \ref{prop: L2 characterization} and standard arguments imply that $f_1$ is strictly positive and smooth. Furthermore, we have that $Q(u,u)\geq 0$ for any $u\in C^\infty_c(\Sigma)$ that is $L^2(\Sigma,d\tilde{s}^2)$-orthogonal to $f_1.$

We claim that $Q(u,u)\geq 0$ for any $u\in C^\infty(\bar{\Sigma})$ that is $L^2(\Sigma,d\tilde{s}^2)$-orthogonal to $f_1$. As $k_{\partial\Sigma}\geq0$, \eqref{eq: finite integral of p} implies that $|A_\Sigma|^2$ and $q$ are integrable, thus $Q  (u,u)$ is finite. Note that $uf_1$ is in $L^1(\Sigma,d\tilde{s}^2)$ since $u, f_1\in L^2(\Sigma,d\tilde{s}^2)$. The proof of the claim follows from the standard cut-off arguments (see e.g. \cite{Rotore-Index-one}*{Prop. 1.1}).

We may view $\bar{\Sigma}$ as a compact domain of a closed orientable surface $\hat{\Sigma}$ of genus $g$ by sticking disks to the boundary components of $\bar{\Sigma}$. Let $\Phi=(\Phi_1,\Phi_2,\Phi_3):\hat{\Sigma}\rightarrow \mathbb{S}^2$ be a holomorphic map with degree less than or equal to $\lfloor\frac{g+3}{2}\rfloor$. Since $(\Sigma,d\tilde{s}^2)$ has finite area it follows that $f_1$ is integrable, so there exists a conformal diffeomorphism $\Psi$ of $\mathbb{S}^2$ such that the composed map $\Psi\circ \Phi$ satisfies
\[\int_{\tilde{\Sigma}}f_1(\Psi\circ \Phi)=\vec{0}.\]
This is usually called the Hersch balancing trick, see \cite{Chen-Fraser-Pang}*{Lemma 5.1} for a proof. It then follows that for $i=1,2,3$
\[\int_\Sigma |\nabla (\Psi \circ \Phi)_i|^2-|A_\Sigma|^2(\Psi \circ \Phi)_i^2-\int_{\partial\Sigma}q(\Psi \circ \Phi)_i^2\geq 0.\]
Summing over $i=1,2,3$ yields
\begin{align*}
    8\pi \operatorname{deg}(\Psi)&=\int_\Sigma |\nabla (\Psi \circ \Phi)|^2\\
    &\geq \int_\Sigma |A_\Sigma|^2+\int_{\partial \Sigma}q\\
    &=\int_\Sigma H_\Sigma^2-2K_\Sigma+\int_{\partial \Sigma} H_{\partial M}- \kappa_{\partial\Sigma}\\
    &\geq -\int_\Sigma 2K_\Sigma-
    \int_{\partial \Sigma}2\kappa_{\partial\Sigma}.
\end{align*}
Note that $K_\Sigma$ and $k_{\partial\Sigma}$ are both integrable due to \eqref{eq: finite integral of p}), then combining above inequality with the Cohn-Vossen inequality (see \cite{Shiohama-The-geometry-of-total-curvature}*{Theorem 2.2.1}) we obtain
\[8\pi\left\lfloor\frac{g+3}{2}\right\rfloor\geq -2\pi(4-4g-2r-k-l).\]
This completes the proof.
\end{proof}

\section{Curvature bounds for capillary surfaces}\label{sec: curv bounds}

In this section, we obtain curvature estimates for strongly stable capillary surfaces. A 3-manifold $M$ with smooth boundary has \textit{bounded geometry} if
there are positive constants $\iota$ and $\Lambda$ such that:
\begin{enumerate}[(i)]
\item $M$ has absolute sectional curvature bound $|K_M| \le \Lambda$;
\item the boundary of $M$ has second fundamental form bound  
	$|h_{\partial M}| \le \Lambda$;
\item every geodesic of $M$ and every geodesic of $\partial M$ with length at most $\iota$ is minimizing;
\item there is a collar neighborhood $U$ of $\partial M$ in $M$ and a function $f \in C^2(U)$ such that $f|_{\partial M} =0$, $|\nabla f| = 1$ and $f(U) \supset [0, \iota)$.
\end{enumerate}
In this case, we will say that $M$ has \textit{curvature bounded above by} $\Lambda$
and \textit{injectivity radius bounded below by} $\iota$. In a local sense, these quantities control  
how far $M$ is from looking like a half-space of $\R^3$. This fact together with the uniqueness theorem for the 
half-plane (Theorem \ref{maintheorem1}) is essential to the proof of the following Theorem.

\begin{customthm}{\ref{curvature bound-introduction}}
Let $\theta \in (0, \pi)$. Then there is a constant $C = C(\theta)$ such that the following holds:
Let $M$ be a 3-manifold with smooth boundary. Assume that $M$ has
curvature bounded above by $\Lambda$ and injectivity radius bounded bellow by $\iota$.
Let $\Sigma$ be a strongly stable edged capillary surface immersed in $M$ at a constant angle $\theta$. Then, 
\begin{itemize}
    \item if
$H_\Sigma \cos \theta \ge 0$, we have
\[|A_{\Sigma}(p)|\min\{d_{\Sigma}(p, \partial\Sigma \setminus 
\partial M), \iota, (\sqrt{\Lambda})^{-1}\}\le C\]
for all $p \in \Sigma$; and
\item if 
$H_\Sigma \cos \theta < 0$, we have
\[|A_{\Sigma}(p)|\min\{d_{\Sigma}(p, \partial\Sigma \setminus 
\partial M), \iota, (\sqrt{\Lambda})^{-1}, -(H_\Sigma\cos \theta)^{-1}\}\le C\]
for all $p \in \Sigma$.
\end{itemize}
\end{customthm}
\begin{proof}
The arguments of Anderson et al.\cite{Anderson-Boundary-regularity-for-ricci-eq}*{Theorem 3.2.1} show that 
there are coordinate charts around any point $p \in M$ such that the
following holds: For any $\alpha \in (0, 1)$ and $Q_0 < \infty$, there is a 
$r_0 > 0$ depending only on $\Lambda, \iota, \alpha, Q_0$ such that:
\begin{enumerate}[(i)]
\item If $d_M(p, \partial M) > r_0$, then there is a neighborhood
	$U$ of $\vec 0$ in $\R^3$ and a coordinate chart 
	$\varphi: U \to B^M_{r_0}(p)$
	such that $\varphi(\vec 0) = p$, and in these coordinates,
	\begin{equation}\label{eq: C0 norms metric}
		Q_0^{-1} \delta_{ij}\le g_{ij} \le Q_0\delta_{ij}
	\end{equation}
	as quadratic forms, and the metric has H\"{o}lder bound
	\begin{equation}\label{eq: C1a norm metric}
		\|g_{ij}\|_{C^{1,\alpha}} \le Q_0.
	\end{equation}
\item If $d_M(p, \partial M) \le r_0$, there is a $b \leq 0$, a 
	neighborhood $U$ of $\vec 0$ in $\{x \in \R^3 : x_1 \ge b\}$ and a coordinate
	chart $\varphi: U \to B^M_{2r_0}(p)$ such
	that $U \cap \{x_1 = b\}$ is mapped to $\partial M$, $\varphi(\vec 0) = p$
	and \eqref{eq: C0 norms
	metric}--\eqref{eq: C1a norm metric} still hold in these coordinates.
\end{enumerate}

From now on fix any $\alpha \in (0, 1), Q_0 < \infty$ and fix $r_0$ such that the above
statements hold for $\iota = \Lambda =1$. We will refer to these coordinates as harmonic coordinates.

Assume by contradiction that there is a sequence of 3-manifolds with bounded
geometry $M_n$, a sequence of strongly stable edged capillary surfaces
$\Sigma_n$ immersed in $M_n$ at constant angle $\theta$, and a sequence of points $q_n \in \Sigma_n$ such that 
\[|A_n(q_n)|\min\{d_n(q_{n}, \partial\Sigma_n \setminus 
\partial M_n), \iota_n, (\sqrt\Lambda_n)^{-1}, (H_n\cos \theta)_{-}^{-1}\} > n,\]
where
\begin{itemize}
    \item $A_n$, $d_n$ are the second fundamental form and the intrinsic
distance of $\Sigma_n$, respectively;
\item $\iota_n$, $\Lambda_n$ are the injectivity radius and curvature bounds of $M_n$, respectively;
\item $H_n$ is the mean curvature of $\Sigma_n$, $(H_n \cos \theta)_{-}$ is $0$ if 
$H_n\cos \theta \ge 0$ and is $-H_n \cos \theta$ otherwise.
\end{itemize}
Since the quantity $|A_\Sigma(p)|\min\{d_n(p,\partial\Sigma\setminus\partial M),\iota,(\sqrt{\Lambda_n})^{-1}, (H\cos\theta)_{-}^{-1}\}$ is scale invariant, by re-scaling the metric of $M_n$ we can assume that $\iota_n \ge 1$, $\Lambda_n \le 1$, and 
$H_n\cos \theta \ge -1$. Hence, by passing to a subsequence, we can assume
\[|A_n(q_n)|\min\{d_n(q_{n}, \partial\Sigma_n \setminus 
\partial M_n), r_0\} > n.\]
Let $\mathcal D_n$ be the intrinsic disk in
$\Sigma_n$ around $q_n$ with radius $\min\{d_n(q_n, \partial\Sigma_n 
\setminus\partial M_n), r_0\}$. Take an interior point $p_n \in \mathcal D_n$ realizing the maximum
of
\[ |A_n(p)|\min\{d_{\mathcal D_n}(p, \partial \mathcal D_n), r_0\}\]
for all $p \in \mathcal D_n$. Let $\lambda_n = |A_n(p_n)|$ and
$R_n = \min\{d_{\mathcal D_n}(p_n, \partial \mathcal D_n), r_0\}$. Note that by
construction
\begin{equation}\label{eq: dist to boundary}
	\lambda_n R_n > n
\end{equation}
and $\lambda_n > n/r_0 \rightarrow \infty$ as $n\rightarrow \infty$. 

Let $\tilde{\mathcal D}_n$ be the disk around
$p_n$ of radius $R_n/2$, note that for all $p \in \tilde{\mathcal D}_n$ have that
$$\min\{d_{\mathcal D_n}(p, \partial \mathcal D_n), r_0\} \ge \frac{R_n}{2}.$$
Hence for all such $p$
\begin{equation}\label{eq: relative curvature bounds}
|A_n(p)| \le \frac{\lambda_n R_n}{\min\{d_{\mathcal D_n}(p, \partial \mathcal D_n), r_0\}}
\le 2\lambda_n.
\end{equation}

Let $\varphi_n: U_n\subset \mathbb{R}^3 \to B^{M_n}_{r_0}(p_n)$ be harmonic coordinates (here by $p_n
\in M_n$ we means the image of $p_n \in \Sigma_n$ by the immersion
$\Sigma_n \looparrowright M_n$).
Define $ \hat{\mathcal D}_n \subset U_n$ to be the pre-image of
$\tilde{\mathcal D}_n$ by $\varphi_n$. We will refer to the point in
$\hat{\mathcal D}_n$ associated to $p_n \in \tilde{\mathcal D}_n$ as $0_n$. Let
\[\lambda_n \hat{\mathcal D}_n \subset (\lambda_n U_n, 
\mu_{\lambda_n^{-1}}^*\varphi_n^*g_n)\]
be $\hat{\mathcal D}_n$ re-scaled by $\lambda_n$. Here $\mu_{\lambda_n^{-1}}: \R^3 \to
\R^3$ is dilation by $\lambda_n^{-1}$ about the origin and $g_n$ is the Riemannian metric in $M_n$. Using the $C^{1,\alpha}$
bounds on the metric tensor \eqref{eq: C0 norms metric} and \eqref{eq: C1a norm
metric} together with the fact that $\lambda_n \to \infty$,
it is easy to show (see \cite{Rosenberg-General-Curvature-Estimates}*{pages 631-632}) that
$(\lambda_nU_n, \mu_{\lambda_n^{-1}}^*\varphi_n^*g) \to (\Half, g_{eucl})$
uniformly on compact sets in the $C^{1,\alpha}$ topology, where $g_{eucl}$ is
the Euclidean metric and $\Half$ is either, some half space $\{x_1 \ge b\} \subset \R^3$ for
some $b \le 0$ or $\R^3$ itself.

From \eqref{eq: C0 norms metric} we can deduce that for any $m \in \N$ ,
there is a $n \in \N$ such that the Euclidean ball of radius $m$,
$B_m(\vec 0)$ is contained in $\lambda_n U_n$. Following 
\cite{Rosenberg-General-Curvature-Estimates},
we define $\Delta_{n,m}$ to be the component of 
$\lambda_n \hat{\mathcal D}_n \cap B_m(\vec 0)$
containing $0_n$. By \eqref{eq: dist to boundary}, the intrinsic distance from 
$0_n \in \lambda_n \hat{\mathcal D}_n$ to its edge is at least $n/2$, and by
\eqref{eq: relative curvature bounds} the norm of the second fundamental form of 
$\lambda_n \hat{\mathcal D}_n$ in the pull-back metric is bounded above by two.
Hence, making $n = n_m$ large enough in relation to $m$, we can assume
that the edge of $\Delta_{n_m,m}$ is disjoint from the interior of $B_m(\vec 0)$, 
and that $\Delta_{n_m,m}$ has Euclidean second fundamental form with norm bounded above by $5$ everywhere, and bounded
below by $1/2$ at $0_n$. Let $\Delta_m =\Delta_{n_m,m}$, passing to a subsequence, two 
scenarios
can happen: $\{\lambda_n d_{n}(p_n, \partial \Sigma_n)\}$ converges to some nonnegative
number or diverges to infinity.

\vspace{.3cm}
\underline{Case 1:} $\lambda_n d_{n}(p_n, \partial \Sigma_n) \to
\infty$.
\vspace{.3cm}

In this case the boundary of $\lambda_n\hat{\mathcal D}_n$ stays away from the
origin, hence the argument can proceed in a very similar way as in the fixed
boundary case \cite{Rosenberg-General-Curvature-Estimates}. We will explain the
argument in detail, as the second case is done in a similar way.

Up to passing to a subsequence, we can assume that the tangent planes $T_{0_n}\Delta_n$ converge.
Rotate the ambient space so this limit plane is $\{x_3 = 0\}$.
It follows from standard arguments (see
\cite{Rosenberg-General-Curvature-Estimates}*{Proposition 2.3},
\cite{Meeks-Ros-Rosemberg-Global-Theory-of-min-surf}*{Lemma 4.35}) that there
are positive constants $C$ and $\delta$ such that a part of 
$\Delta_n$ is the graph of a function $u_n$ over $D_\delta = \{x \in
B_\delta(\vec 0) : x_3 = 0\}$ and
\[|\nabla u_n| < 1,  \|u_n\|_{C^{2}(D_{\delta})} < C.\]
Applying the Schauder estimates to the mean curvature equation
(see \cite{Rosenberg-General-Curvature-Estimates}*{
Lemma 2.4}) it is easy to see that there is a $\delta' \in (0, \delta)$ such
that, after possibly making $C$ larger
\[\|u_n\|_{C^{2,\alpha}(D_{\delta'})} < C.\]

Passing to a subsequence we can assume there is a $u: D_\delta \to \R$ such
that $u_n \to u$ in $C^{2}(D_{\delta'})$ and $C^{2, \alpha/2}(D_{\delta'})$.
Fix a $\delta'' \in (\delta', \delta)$ and let $x_0 \in D_{\delta'}$  be at
Euclidean distance of $\delta'/2$ from the origin. By repeating the arguments
above we have that for $n$ large, the $\Delta_n$ are graphs over the disk of
radius $\delta''$ in the tangent plane of the graph of $u$ at $(x_0, u(x_0))$.
Passing to a subsequence, these functions must converge in the $C^{2}$ and
$C^{2,\alpha/2}$ topology, extending the graph of $u$ to a larger surface.
Continuing this construction and using a diagonal argument we encounter a
complete noncompact CMC surface $S$ immersed in $(\mathbb{R}^3,g_{eucl})$ with second fundamental
form bounded above by 5 everywhere and passing through the origin with nonzero second
fundamental form.

Let $\tilde S$ be the universal cover of $S$. We will show that $\tilde S$
accepts a positive Jacobi function
\footnote{The definition of the index of a surface
touching the boundary of the manifold in its interior can be ambiguous
\cite{Carlotto-Franz-Inequivalent-complexity-criteria}. For this reason we do
not state that $\tilde S$ is strongly stable.
}. By arguments of Fischer-Colbrie and Schoen
\cite{Fischer-Colbrie-Schoen-The-structure-of-complete-stable}*{Theorem 1} we
only have to show that every bounded open set $\tilde U \subset \tilde S$ accepts
a positive Jacobi function. Then it follows from arguments of Fischer-Colbrie and
Schoen \cite{Fischer-Colbrie-Schoen-The-structure-of-complete-stable}
(see also \cites{doCarmo-Peng, daSilveira-Stability-of-complete}) that $\tilde
S$ must be a flat plane, contradicting the fact that this
surface passes through the origin with nonzero second fundamental form.

Now we will show that $\tilde U$ accepts a positive Jacobi function.
Let $\Pi: \tilde S \times (-\epsilon, \epsilon)$ be the immersion $\Pi(p,t) = 
\tilde f(p) + t\tilde \nu(p)$ where $\tilde f$ is the immersion of $\tilde S$
in $\R^3$ and $\tilde \nu$ is a choice of normal in $\tilde S$. Let $U$ be the
image of $\tilde U$ in $S$. By
construction, for $n$ large, a piece of $\Delta_n$ will be close enough to $U$
everywhere so that this piece is parametrized by a function $\Pi(\cdot, \tilde u_n(\cdot))$
where $\tilde u_n: \tilde U \to (-\epsilon, \epsilon)$. Since $\Delta_n$ is strongly
stable it accepts a positive Jacobi function $v_n$. Let $\tilde v_n: \tilde S \to
\R^+$ be such that $\tilde v_n(\cdot) = v_n(\Pi(\cdot, \tilde u_n(\cdot)))$, similarly
the Jacobi operator on $\Delta_n$ defines an elliptic operator on
$\tilde U$. Since the metrics are converging in $C^{1,\alpha}$ to the Euclidean
metric and $\tilde u_n \to 0$ in $C^{2, \alpha/2}$, we can conclude that
these operators converge to the Jacobi operator on $\tilde U$ and that the functions $\tilde
v_n$ converge to a positive Jacobi function in $\tilde U$.

\vspace{.3cm}
\underline{Case 2:} $\lambda_n d_n(p_n, \partial \Sigma_n)$ converges to some
nonnegetive number.
\vspace{.3cm}

Let $\partial\Half_n \cap \lambda_nU_n$ be the set in $\lambda_n U_n$ associated
to $\partial M$, where $\Half_n$ is a half-space of $\R^3$. By shifting the $\lambda_nU_n$ and 
$\lambda_n\hat{\mathcal D}_n$ slightly we can
assume that $\Half_n = \Half$ for all $n$ large. We will still denote the point
associated to $p_n$ in this shifted surface as $0_n$. Note that $0_n$ might not
map to origin, but their images go to the origin as $n \to \infty$.

Using the bounds on the Euclidean second fundamental form of $\Delta_n$ it follows from standard arguments that there are constants $C$ and $\delta$ depending on $\theta$ such that the following holds:

\begin{itemize}
    \item If $\lim_n d_{\Delta_n}(0_n, \partial \Delta_n) \ge 2\delta$, then we can pass to a subsequence and rotate the ambient space so a piece of $\Delta_n$ near the origin is the graph of a function $u_n$ over $D_\delta$, furthermore $u_n$ has $C^1$ and $C^2$ bounds as in Case 1.
    
    \item Now assume $\lim_n d_{\Delta_n}(0_n, \partial \Delta_n) < 2\delta$.
    Choose a point $x_n \in \partial \Delta_n$ closest to $0_n$, passing to a subsequence and rotating we can assume that the limit of the angle between the normal to $\Delta_n$ at $x_n$ and the third coordinate vector field $E_3$ is
    $\pi/2 -\theta$. Then a piece of $\Delta_n$ near $x_n$ can be parametrized by the graph of a function $u_n$ over $D_{4\delta} \cap \Half$, moreover $u_n$ has $C^1$ and $C^2$ bounds as in Case 1.
\end{itemize}

Using boundary
Schauder estimates \cite{Nirenberg-Estimates-near-boundary}*{Theorem 7.3}
and the fact that these surfaces are capillary it is easy to obtain
$C^{2,\alpha}$ bounds for these graphs analogous to the ones obtained in Case 1.
Using the same construction as in Case 1 we obtain a capillary surface $S$
immersed in $\Half$ at a constant angle $\theta$ such that $ H_S\cos \theta = \lim_n H_n\cos\theta / \lambda_n \ge 0$.

The arguments used in Case 1 together with Proposition \ref{prop: positive jacobi field} allow us to
construct a positive Jacobi function $\tilde v$ on the universal cover 
$\tilde S$ of $S$. The only difference is that, in this case instead of considering the normal $\tilde \nu$ to $\tilde S$ we must consider a nonzero vector filed that is normal to $\tilde S$ away from $\partial \tilde S$ and is parallel to $\partial \Half$ at $\partial \tilde S$. Finally, remark \ref{remark: regidity compact case} we conclude that $\tilde S$ must be noncompact, hence it follows from Theorem \ref{maintheorem1} that it is a plane, contradicting the fact that it passes through the origin with nonzero second fundamental form.
\end{proof}

\begin{remark}
In the case where $\Sigma$ is capillary, that is, has empty edge we interpret 
$d_\Sigma(p, \partial \Sigma \setminus \partial M)$ to be positive infinity for all $p \in \Sigma$.
\end{remark}

\begin{corollary}\label{cor:diameter bounds}
Fix $\theta \in (0, \pi /2]$ and take $C, M, \iota$ and $\Lambda$ as in the statement of Theorem
\ref{curvature bound-introduction}. Let $\Sigma$ be an edged capillary surface with finite index immersed immersed
in $M$ at constant angle $\theta$. Assume $H_\Sigma \ge \sqrt 2C\max\{\iota^{-1},\sqrt \Lambda\}$. Then,
\begin{itemize}
    \item in case $\Sigma$ has nonempty edge,
\begin{equation}\label{eq: diameter bound edged}
   d_\Sigma(p, \partial \Sigma \setminus \partial M) 
   \le  (\Index_s(\Sigma) + 1) \frac{2\sqrt 2 C}{H_\Sigma}
\end{equation}
for all $p \in \Sigma$; and
\item in case $\Sigma$ is capillary
\begin{equation}\label{eq: diameter bound}
   \textup{diam}_\Sigma(\Sigma)
   \le \Index_s(\Sigma)\frac{2\sqrt 2 C}{H_\Sigma},
\end{equation}
where $\textup{diam}_\Sigma(\Sigma)$ is the intrinsic diameter of $\Sigma$.
\end{itemize}
\end{corollary}
\begin{proof}
We will only show \eqref{eq: diameter bound edged} since the proof of \eqref{eq: diameter bound}
is very similar. Let $I = \Index_s(\Sigma)$, first we consider the case where $I = 0$. Since $|A_\Sigma| \ge H_\Sigma/\sqrt 2$, it follows
from Theorem \ref{curvature bound-introduction} that 
\[H_\Sigma d_\Sigma(p, \partial \Sigma \setminus \partial M) \le \sqrt 2 C \]
for all $p \in \Sigma$.

Now consider the case when $I > 0$. Let $p \in \Sigma$ and take $\gamma: [0, L] \to \Sigma$ to be an unit speed
geodesic minimizing the distance from $p$ to $\partial \Sigma \setminus \partial M$. Assume by contradiction 
that $\gamma$ has length $(I + 1) 2\sqrt 2 C/H_\Sigma + (I + 1)\epsilon$ for some $\epsilon > 0$.
Fix numbers $0 = t_0 < t_1 \cdots < t_{I + 1} = L$ such that $t_{i+ 1} - t_i \ge  2\sqrt 2 C/H_\Sigma + \epsilon$ for $i = 0, \cdots, I$. 
Let $\mathcal D_i$ be the geodesic disks around $\gamma(t_i)$ with radius $\sqrt 2C/H_\Sigma + \epsilon/2$.
Note that by construction these disks are pair-wise disjoint and they are disjoint from the edge of $\Sigma$
for $i = 0, \cdots, I$. The case $I = 0$ implies that all these disks must be strongly unstable. This means that $\Sigma$ must have index at least $I + 1$, a contradiction.
\end{proof}
\vspace{1cm}

\appendix
\section{A Poincar\'e-type inequality}\label{appendix A}

In this appendix we will show a Poincar\'e-type inequality that plays a role in the
proof of Proposition \ref{prop: L2 characterization}. We assume that $w$ is a
nonnegative Lipschitz function on $\Sigma$ with bounded support. Note that $w$ can be positive on
the boundary of $\Sigma$.

Let $\Omega$ be the support of $w$ and 
$\Gamma = \Omega \cap \partial\Sigma$. For a function $\varphi \in C(\Omega)$
define the norms
\[\|\varphi\|_\Omega^2 = \int_\Omega \varphi^2,\ \|\varphi\|_{\Omega,w}^2 =
\int_\Omega w\varphi^2.\]
In order to show our Poinca\'e-type inequality we need:
\begin{proposition}
There is a constant $C > 0$ such that
\begin{equation}
\label{eq: bound weighted trace}
\int_\Gamma w\varphi^2
\le C \|\varphi\|_{\Omega}(\|\varphi\|_{\Omega}+ \||\nabla
\varphi|\|_{\Omega,w})
\end{equation}
for all $\varphi \in C^\infty(\Sigma)$.
\end{proposition}
\begin{proof}
Let $\tilde \eta$ be a compactly supported smooth vector field in $\Sigma$ that extends
the conormal $\eta$. Then for $\varphi \in
C^\infty(\Sigma)$ we have that
\begin{align*}
\int_\Gamma w\varphi^2
&= \int_\Omega \text{div}(w\varphi^2\tilde \eta)\\
&= \int_\Omega \varphi^2 \ll\nabla w, \tilde \eta \rl + 2w\varphi \ll \nabla
\varphi, \tilde \eta\rl + w\varphi^2\text{div}(\tilde \eta)\\
&\le  2\left(\int_\Omega w\varphi^2\right)^\frac{1}{2}
\left( \int_\Omega w\ll \nabla \varphi, \tilde \eta\rl^2\right)^\frac{1}{2}
+ \int_\Omega \varphi^2 (|\nabla w||\tilde \eta| + w|\text{div}(\tilde
\eta)|)\\
&\le C \|\varphi\|_{\Omega}(\|\varphi\|_{\Omega}+ \||\nabla
\varphi|\|_{\Omega,w}),
\end{align*}
where $C$ depends only on $w$ and $\tilde \eta$.
\end{proof}

\begin{lemma}\label{lemma: poincare}
For every $\epsilon > 0$ there is a constant $C = C(\epsilon) > 0$
such that
\[\int_{\Gamma} w \varphi^2 \le \epsilon \int_\Omega w|\nabla
\varphi|^2 + C\int_{\Omega} \varphi^2\]
for all $\varphi \in C^\infty(\Sigma)$.
\end{lemma}
\begin{proof}
Suppose by contradiction that the result does not hold. Then, for some
$\epsilon_0 >0$ there is a sequence $\varphi_n \in C^\infty(\Sigma)$ such that
for all $n \in \N$
\begin{equation}\label{eq: assumption poincare}
\int_{\Gamma} w \varphi_n^2 > \epsilon_0 \int_\Omega w|\nabla
\varphi_n|^2 + n\int_{\Omega} \varphi_n^2.
\end{equation}
By rescaling we can assume that $\||\nabla \varphi_n\|_{\Omega, w} = 1$.

Firstly, if $\|\varphi_n\|_\Omega \to \infty$, then by rescaling this
sequence we can obtain another sequence of functions $\psi_n \in C^\infty(\Sigma)$ such that
$\||\nabla \psi_n|\|_{\Omega, w} \to 0$, $\|\psi_n\|_{\Omega} = 1$ and
$\int_{\Gamma} w \psi_n^2 > n$. This contradicts \eqref{eq: bound weighted trace}.

Thus we can assume that $\|\varphi_n\|_{\Omega}$ are uniformly
bounded above. It follows from \eqref{eq: bound weighted trace} that 
$\int_{\Gamma} w \varphi_n^2$ are uniformly bounded above, and hence by
\eqref{eq: assumption poincare} we must have that $\|\varphi_n\|_{\Omega} \to
0$. This is impossible since, by \eqref{eq: bound weighted trace} it would imply
that $\int_{\Gamma} w \varphi_n^2 \to 0$ but $\int_{\Gamma} w \varphi_n^2 >
\epsilon_0$.
\end{proof}

\bibliography{refs}
\end{document}